\documentclass[a4paper,11pt]{article}

\usepackage[left=2.7cm,right=2.7cm,top=3.5cm,bottom=3cm]{geometry}

\usepackage{amsmath,amssymb,amsthm,amscd,amsfonts}
\usepackage{latexsym}
\usepackage{graphicx}
\usepackage[all]{xy}
\usepackage{color}
\usepackage[T1]{fontenc}
\usepackage{authblk}

\newcommand{\Z}{{\mathbb Z}}
\newcommand{\Q}{\mathbb Q}

\newcommand{\E}{{\mathcal{E}}}

\newtheorem{theorem}{Theorem}[section]
\newtheorem{lemma}[theorem]{Lemma}
\newtheorem{corollary}[theorem]{Corollary}
\newtheorem{proposition}[theorem]{Proposition}

\theoremstyle{definition}
\newtheorem{definition}[theorem]{Definition}
\newtheorem{remark}[theorem]{Remark}
\newtheorem{example}[theorem]{Example}

\newcommand{\ds}{\displaystyle}

\newcommand{\z}{{\zeta}}
\newcommand{\ov}{\overline}

\DeclareMathOperator{\Gal}{\rm Gal}
\DeclareMathOperator{\Id}{\rm Id}
\DeclareMathOperator{\GL}{{\rm GL}}
\DeclareMathOperator{\SL}{{\rm SL}}

\newcommand{\s}{\sigma}

\renewcommand{\b}{\beta}
\renewcommand{\c}{\gamma}
\renewcommand{\d}{\delta}
\newcommand{\D}{\Delta}
\renewcommand{\leq}{\leqslant}
\renewcommand{\geq}{\geqslant}

\title{Fields generated by torsion points of elliptic curves}

\author[1]{Andrea Bandini}
\affil[1]{Dipartimento di Matematica e Informatica, Universit\`a degli Studi di Parma\\
Parco Area delle Scienze, 53/A - 43124 Parma, Italy, e-mail: andrea.bandini@unipr.it}

\author[2]{Laura Paladino\thanks{L. Paladino is partially supported by Istituto Nazionale di Alta Matematica,
grant research \emph{Assegno di ricerca Ing. G. Schirillo}, and partially supported by the European Commission
and by Calabria Region through the European Social Fund.}}
\affil[2]{Dipartimento di Matematica, Universit\`a di Pisa\\
Largo Bruno Pontecorvo, 5 - 56127 Pisa, Italy, e-mail: paladino@mail.dm.unipi.it}

\date{ }

\begin{document}

\maketitle


\bigskip


\begin{abstract}
Let $K$ be a field and let $\E$ be an elliptic curve defined over $K$.
Let $m$ be a positive integer, prime with ${\rm char}(K)$ if ${\rm char}(K)\neq 0$;
we denote by $\E[m]$ the $m$-torsion subgroup of $\E$ and
by $K_m:=K(\E[m])$ the field obtained by adding to $K$ the coordinates
of the points of $\E[m]$. Let $P_i:=(x_i,y_i)$ ($i=1,2$) be a $\Z$-basis for $\E[m]$; then
$K_m=K(x_1,y_1,x_2,y_2)$. We look for small sets of generators for $K_m$ inside
$\{x_1,y_1,x_2,y_2,\z_m\}$ trying to emphasize the role of $\z_m$ (a primitive $m$-th root of unity).
In particular, we prove that $K_m=K(x_1,\z_m,y_2)$, for any odd $m\geq 5$.
When $m$ is prime and $K$ is a number field we prove that the generating set $\{x_1,\z_m,y_2\}$ is often minimal.
We also describe explicit generators, degree and Galois groups of
the extensions $K_m/K$ for $m=3$ and $m=4$, when ${\rm char}(K)\neq 2,3$.
\end{abstract}

\maketitle

\section{Introduction} \label{sec1}
Let $K$ be a field of any characteristic and
let $\E$ be an elliptic curve defined over $K$. Let $m$ be a positive integer, prime with
${\rm char}(K)$ if ${\rm char}(K)\neq 0$. We denote by $\E[m]$ the $m$-torsion subgroup of $\E$ and by $K_m:=K(\E[m])$
the field generated by the points of $\E[m]$, i.e. the field obtained by adding to $K$ the coordinates
of the $m$-torsion points of $\E$. As usual, for any point $P\in \E$, we let $x(P)$, $y(P)$ be its coordinates
and we indicate its $m$-th multiple simply by $mP$. We denote by $\{P_1\,,P_2\}$ a $\Z$-basis for $\E[m]$;
then $K_m=K(x(P_1),x(P_2),y(P_1),y(P_2))$. To ease notation, we put $x_i:=x(P_i)$ and $y_i:=y(P_i)$ ($i=1,2$).
By Artin's primitive element theorem the extension $K_m/K$ is monogeneous and one can find a unique generator
for $K_m/K$ by combining the above coordinates. On the other hand, by the properties of the Weil pairing $e_m$,
we have that $e_m(P_1,P_2)\in K_m$ is a primitive $m$-th root of unity (we denote it by $\z_m$). We want to
emphasize the importance of $\zeta_m$ as a generator of $K_m/K$ and look for minimal (i.e., with the smallest
number of elements) sets of generators contained in $\{x_1,x_2,y_1,y_2,\z_m\}$. This kind of information is useful
for describing the fields in terms of degrees and Galois groups, as we shall explicitly show
for $m=3$ and $m=4$, when ${\rm char}(K)\neq 2,3$. Other applications are local-global problems (see, e.g., \cite{DZ1}
or the particular cases of \cite{Pal} and \cite{Pal2}), descent problems (see, e.g., \cite{SS} and the references there
or, for a particular case, \cite{Ba} and \cite{Ba2}), Galois representations,
points on modular curves (see Section \ref{SecGal}) and points on Shimura curves.

It is easy to prove that $K_m= K(x_1,x_2,\z_m,y_1)$ (see Lemma \ref{zetam})
and we expected a close similarity between the roles of the
$x$-coordinates and $y$-coordinates; this turned out to be true in relevant cases. Indeed in Section \ref{MinSet}
(mainly by analysing the possible elements of the Galois group $\Gal(K_m/K)\,$) we prove that $K_m=K(x_1,\z_m,y_1,y_2)$
at least for odd $m\geq 5$. This leads to the following (for more precise and general statements see
Theorems \ref{dihedral}, \ref{ordinates1} and \ref{ordinates1++})

\begin{theorem}\label{IntroThm1}
If $m\geqslant 3$, then $K_m=K(x_1+x_2,x_1x_2,\z_m,y_1)$. Moreover if $m\geqslant 4$, then
\[ K_m=K(x_1,\z_m,y_1,y_2) \Longrightarrow K_m=K(x_1,\z_m,y_2)\ .\]
In particular $K_m=K(x_1,\z_m,y_2)$ for any odd integer $m\geq 5$.
\end{theorem}

Note that, by Theorem \ref{IntroThm1}, we have $K_p=K(x_1,\z_p,y_2)$, for any prime $p\geq 5$.
The set $\{x_1,\z_p,y_2\}$ seems a good candidate (in general) for a minimal set of generators
for $K_p/K$. Indeed, when $K$ is a number field and $\E$ has no complex multiplication, by Serre's open image theorem
(see, e.g., \cite[Appendix C, Theorem 19.1]{Sil}), we expect that the natural representation
\[ \rho_{\E,p}:\Gal(\ov{K}/K) \rightarrow  \GL_2(\Z/p\Z) \]
provides an isomorphism $\Gal(K_p/K) \simeq \GL_2(\Z/p\Z)$ for almost all primes $p$,
and there are hypotheses on $x_1$, $\z_m$ and $y_2$ (see Theorem \ref{GalGL}) which guarantee that
\[ [K(x_1,\z_m,y_2):K]=(p^2-1)(p^2-p)=| \GL_2(\Z/p\Z)| \ .\]
For (almost all) the {\em exceptional primes} for which $|\Gal(K_p/K)|<|\GL_2(\Z/p\Z)|$ (see Definition \ref{ExcPrimes}),
we employ some well known results on Galois representations and on subgroups of $\GL_2(\Z/p\Z)$ to reduce further the set of generators.
Joining the results of Lemma \ref{LVLemma}, Theorem \ref{Borel} and Theorem \ref{Cartan} we obtain

\begin{theorem}\label{IntroThm2}
Let $K$ be a number field. Assume that $p\geq 53$ is unramified in $K/\Q$ and exceptional for the curve $\E$. Then \begin{itemize}
\item[{\bf 1.}] $p\equiv 2\pmod 3 \Longrightarrow K_p=K(\z_p,y_2)$;
\item[{\bf 2.}] $p\equiv 1\pmod 3 \Longrightarrow [K_p:K(\z_p,y_2)]$ is $1$ or $3$.
\end{itemize}
\end{theorem}

In Subsection \ref{SecMod} we give just a hint of the possible applications to points of modular curves. Similar
applications, even to Shimura curves, can be further developed in the future. Modular curves
might provide a different approach (and more insight) to problems  analogous to those treated here.

The final sections are dedicated to the cases $m=3$ and $m=4$, when ${\rm char}(K)\neq 2,3$. We use the explicit formulas
for the coordinates of the torsion points to give more information on the extensions $K_3/K$ and $K_4/K$,
such as their degrees and their Galois groups.
\bigskip

\noindent {\bf Acknowledgement.}
The authors would like to express their gratitude to Antonella Perucca for suggesting the topic of a generalization
of the results of \cite{BP} and for providing several suggestions, comments and improvements on earlier drafts of
this paper. %\color{blue} The author thank an anonymous referee for having suggested to state the results of the
%first two sections for a general field $K$, instead of a number field, and to apply their results to modular curves
%and Shimura curves. \normalcolor

\section{The equality $K_m=K(x_1+x_2,x_1x_2,\z_m,y_1)$} \label{prelim}
As mentioned above, we consider a field $K$ of any characteristic and an elliptic curve $\E$ defined over $K$.
Throughout the paper we always assume that $m$ is an integer, $m\geqslant 2$ and, if ${\rm char}(K)\neq 0$, that
$m$ is prime with ${\rm char}(K)$. We choose two points $P_1=(x_1,y_1)$ and $P_2=(x_2,y_2)$
which form a $\Z$-basis of the $m$-torsion subgroup $\E[m]$ of $\E$.
We define $K_m:=K({\E}[m])$ and we denote by $K_{m,x}$ the extension of $K$ generated by the
$x$-coordinates of the points in $\E[m]$. So we have
\[ K(x_1,x_2)\subseteq K_{m,x}\subseteq K_m=K(x_1,x_2,y_1,y_2)\ .\]
Let $e_m: \E[m]\times \E[m] \longrightarrow {\boldsymbol\mu}_m$ be the Weil Pairing, where ${\boldsymbol\mu}_m$ is
the group of $m$-th roots of unity. By the properties of $e_m\,$,
we know that ${\boldsymbol\mu}_m\subset K_m$ and, once $P_1$ and $P_2$ are fixed, we put $e_m(P_1,P_2)=:\z_m$
(a primitive $m$-root of unity). We remark that the choice of $P_1$ and $P_2$ is arbitrary; we use this convention
for $\z_m\,$ (which obviously has no effect on the generated field since $K(\z_m)=K({\boldsymbol\mu}_m)$
for any primitive $m$-th root of unity) to simplify notations and computations.
In particular for any $\s\in \Gal(K_m/K)$, we have
\[ \s(\z_m)=\s(e_m(P_1,P_2))=e_m(P_1^\s,P_2^\s)=\z_m^{\det(\s)}, \]
where, for simplicity, we still use $\s$ to denote the matrix $\rho_{\E,m}(\s)\in \GL_2(\Z/m\Z)\,$.

The next lemma is rather obvious, but it shows how $\zeta_m$ can play the role of one of the $y$-coordinates
in generating $K_m$ and it will be useful in the rest of the paper.

\begin{lemma} \label{zetam}
We have $K_m= K(x_1,x_2,\z_m,y_1)$.
\end{lemma}

\begin{proof} An endomorphism of $\E[m]$ fixing  $P_1$ and $x_2$ is of type
$\s=\left( \begin{array}{cc} 1 & 0 \\ 0 & \pm 1 \\ \end{array} \right)$. If it also fixes $\z_m$, then $\det(\s)=1$ and
eventually $\s=\Id$.
\end{proof}

We now show that $\z_m$ and $y_1y_2$ are closely related over the field $K(x_1,x_2)$.
Let $(x_3,y_3)$ (resp. $(x_4,y_4)\,$) be the coordinates of the point $P_3:=P_1+P_2$ (resp. $P_4:=P_1-P_2\,$).
By the group law of $\E$, we may express $x_3$ and $x_4$ in terms of $x_1\,$,
$x_2\,$, $y_1$ and $y_2\,$:
\begin{equation} \label{eqx34}
x_3= \frac{(y_1-y_2)^2}{(x_1-x_2)^2} - x_1 -x_2
\qquad{\rm and}\qquad
x_4= \frac{(y_1+y_2)^2}{(x_1-x_2)^2} - x_1 -x_2
\end{equation}
(note that $x_1\neq x_2$ because $P_1$ and $P_2$ are independent).
By taking the difference of these two equations we get
\begin{equation} \label{eqy1y2}
y_1y_2=\frac{(x_4-x_3)(x_1-x_2)^2}{4} \ .
\end{equation}

\begin{lemma}\label{34versusy}
We have $K(x_1,x_2, y_1y_2)=K(x_1,x_2,x_3,x_4)$ and $K_m=K_{m,x}(y_1)$
\end{lemma}

\begin{proof}
Since $y_i^2\in K(x_i)$, equations \eqref{eqx34} and \eqref{eqy1y2} prove the first equality.
For the final statement just note that $K_m=K_{m,x}(y_1,y_2)=K_{m,x}(y_1)$.
\end{proof}

\noindent More precisely, we have

\begin{lemma}\label{casi1}
Let $L=K(x_1,x_2)$. Exactly one of the following cases holds:
\begin{itemize}
\item[{\bf 1.}] $[K_m: L]=1$;
\item[{\bf 2.}] $[K_m: L]=2$ and $L(y_1y_2)=K_m\,$;
\item[{\bf 3.}] $[K_m: L]=2$, $L= L(y_1y_2)$ and $L(y_1)= L(y_2)=K_m\,$;
\item[{\bf 4.}] $[K_m: L]=4$ and $[L(y_1y_2):L]=2$.
\end{itemize}
\end{lemma}

\begin{proof}
Obviously the degree of $K_m$ over $L$ divides $4$. If $[K_m: L]=1$, then we are in case {\bf 1}.
If $[K_m: L]=4$, then  $y_1$ and $y_2$ must generate different quadratic extensions of $L$ and so $[L(y_1y_2):L]=2$
and we are in case {\bf 4}.
If $[K_m: L]=2$ and $y_1y_2\notin L$, then we are in case {\bf 2}. Now suppose that $[K_m: L]=2$ and $y_1y_2\in L$.
Then $y_1$ and $y_2$ generate the same extension of $L$ and this extension is nontrivial,
so we are in case {\bf 3}.
\end{proof}

\begin{lemma}\label{zetafuoriL}
If $y_1y_2\notin K(x_1,x_2)$, then $\z_m\notin K(x_1,x_2)$.
\end{lemma}

\begin{proof} We are in case {\bf 2} or case {\bf 4} of Lemma \ref{casi1} and, in particular, $m>2$ because
of $K_2=L$. We have $[L(y_1y_2):L]=2$ and there exists
$\tau \in \Gal(K_m/L)$ such that $\tau(y_1y_2)=-y_1y_2\,$. Without
loss of generality, we may suppose $\tau(y_1)=-y_1$ and
$\tau(y_2)=y_2$ so that $\tau=\left( \begin{array}{cc} -1 & 0 \\ 0 & 1 \\ \end{array} \right)$ and
$\tau(\z_m) =\z_m^{-1}\,$.
Since $m\neq 2$, $\z_m^{-1}\neq \z_m$ and we get $\z_m\notin L$.
\end{proof}

\noindent The connection between $\z_m$ and $y_1y_2$ is provided by the following statement.

\begin{theorem}\label{zetaversusy}
We have
\[ K(x_1,x_2, \z_m)=K(x_1,x_2, y_1y_2) \ .\]
\end{theorem}

\begin{proof} We first prove that $\z_m\in K(x_1,x_2,y_1y_2)$ by
considering the four cases  of Lemma \ref{casi1}.

\noindent\emph{Case {\bf 1} or {\bf 2}}: we have $K(x_1,x_2, y_1y_2)=K_m$ so the statement clearly holds.

\noindent\emph{Case {\bf 3}}: we have $K_m=L(y_1)$ and $y_1y_2\in L$ so the nontrivial element $\tau\in \Gal(K_m/L)$ maps $y_i$
to $-y_i$ for $i=1,2$. In particular, $\tau=-\Id$ and $\tau(\z_m)=\z_m\,$.
Hence $\z_m\in L=K(x_1,x_2)$.

\noindent\emph{Case {\bf 4}}: since $K_m=L(y_1,y_2)$ and
$\Gal(K_m/L)\simeq \Z/2\Z \times \Z/2\Z$, there exists $\tau\in \Gal(K_m/L)$ such that $\tau(y_i)=-y_i$ for $i=1,2$.
The field fixed by $\tau$ is $L(y_1y_2)$ and, as in the previous case, we get $\tau(\z_m)=\z_m\,$:
so $\z_m\in L(y_1y_2)=K(x_1,x_2, y_1y_2)$.

\noindent Now the statement of the theorem is clear if we are in case {\bf 1} or in case {\bf 3} of Lemma \ref{casi1}.
In cases {\bf 2} and {\bf 4} of Lemma \ref{casi1} we have $[L(y_1y_2):L]=2$, $\z_m\notin L$ (Lemma \ref{zetafuoriL})
and $L(\z_m)\subseteq L(y_1y_2)$. These three facts yield $L(\z_m)=L(y_1y_2)$.
\end{proof}

\noindent We conclude this section with the equality appearing in the title, which still focuses
more on the $x$-coordinates. For that we shall need the following lemma.

\begin{lemma}\label{twist}
The extension $K(x_1,x_2)/K(x_1+x_2,x_1x_2)$ has degree $\leqslant 2$. Its Galois group
is either trivial or generated by $\s$ with $\s(x_i)=x_j$ ($i\neq j$).
\end{lemma}

\begin{proof}
Just note that $x_1$ and $x_2$ are the roots of $X^2-(x_1+x_2)X+x_1x_2\,$.
\end{proof}

\begin{corollary}\label{RealZeta2}
We have $K(\z_m+\z_m^{-1})\subseteq K(x_1+x_2,x_1x_2)$.
\end{corollary}

\begin{proof}
This is obvious if $K(x_1,x_2)=K(x_1+x_2,x_1x_2)$. If they are different, take
the nontrivial element $\s$ of $\Gal(K(x_1,x_2)/K(x_1+x_2,x_1x_2))$. By Lemma \ref{twist},
we have $\s(P_i)=\pm P_j$ ($i\neq j$), hence $\det(\s)=\pm 1$.
\end{proof}

\begin{theorem}\label{dihedral}
For $m\geqslant 3$ we have $K_m=K(x_1+x_2,x_1x_2,\z_m,y_1)$.
\end{theorem}

\begin{proof}
We consider the tower of fields
\[ K(x_1+x_2,x_1x_2) \subseteq K(x_1,x_2) \subseteq K(x_1,x_2,\z_m,y_1)=K_m \]
and adopt the following notations:
\[ \begin{split} G&:=\Gal (K_m/K(x_1+x_2,x_1x_2))\ , \\
 H&:=\Gal (K_m/K(x_1,x_2))\vartriangleleft G \ ,\\
G/H& =\Gal (K(x_1,x_2)/K(x_1+x_2,x_1x_2))\ .\\
\end{split} \]
If  $K(x_1+x_2,x_1x_2)=K(x_1,x_2)$, then the statement holds by Lemma \ref{zetam}.\\
By Lemma \ref{twist}, we may now assume that $G/H$ has order $2$ and its nontrivial automorphism swaps $x_1$ and $x_2\,$.
Then there is at least one element $\tau\in G$ such that $\tau(x_i)=x_j$, with $i,j\in\{1,2\}$ and $i\neq j$.
Therefore $\tau(y_i)=\pm y_j$. The possibilities are:
\[ \tau=\pm \tau_1 = \left( \begin{array}{cc} 0 & \pm 1 \\ \pm 1 & 0 \\ \end{array} \right)\
{\rm (of\ order\ 2)\ and}\
\tau=\pm \tau_2 = \left( \begin{array}{cc} 0 & \mp 1 \\ \pm 1 & 0 \\ \end{array} \right)\
{\rm (of\ order\ 4)} \ .\]
Note that $\tau_2^2=-\Id$ fixes both $x_1$ and $x_2\,$, i.e. the generators of the field $L$ of Lemma
\ref{casi1}.  Moreover, if $y_2=\pm y_1\,$, then we have
\[ \tau_2^2(P_1)=\tau_2(P_2)=\tau_2(x_2,\pm y_1)=(x_1,\pm y_2)=P_1, \  \]
a contradiction. The automorphisms $\tau_1$ and $\tau_2$ generate
a non abelian group of order 8 with two elements of order 4, i.e., the dihedral group
\[ D_4=\langle\tau_1\,,\tau_2\,:\,\tau_1^2=\tau_2^4=\Id\ {\rm and}\ \tau_1\tau_2\tau_1=\tau_2^3\rangle \ .\]
So $G$ is a subgroup of $D_4\,$. Since $G/H$ has order $2$, $H$ is isomorphic to either
$1$, $\Z/2\Z$ or $(\Z/2\Z)^2$ (note that $\tau_2\not\in H$) and its nontrivial elements can at most be the following
 \[ \tau_1\tau_2=\tau_2^3\tau_1 = \left( \begin{array}{cc} -1 & 0 \\ 0 & 1 \\ \end{array} \right) \ ,\
\tau_2\tau_1=\tau_1\tau_2^3 = \left( \begin{array}{cc} 1 & 0 \\ 0 & -1 \\ \end{array} \right)
\ {\rm and}\ \  -\Id \ .\]
We distinguish three cases according to the possible degrees $[K_m:K(x_1,x_2)]$ mentioned in Lemma \ref{casi1}.

\noindent \emph{The case $K_m=K(x_1,x_2)$.} Since $|H|=1$ and $|G/H|=2$, then $|G|=2$. The nontrivial automorphism of $G$
has to be $\pm \tau_1\,$. In both cases $G$ does not fix $\z_m\,$: so
$\z_m\in K(x_1,x_2)- K(x_1+x_2,x_1x_2)$ and we deduce $K(x_1+x_2,x_1x_2,\z_m)=K(x_1,x_2)$ $=K_m\,$.

\noindent\emph{The case $[K_m:K(x_1,x_2)]=4$.} Since $|H|=4$ and $|G/H|=2$, we have $G\simeq D_4\,$.
The subgroup $\langle\tau_2\rangle$ of $D_4$ is normal of index $2$ and it does not contain $\tau_1\,$.
Moreover, $\tau_2$ fixes $\z_m$ and $\tau_1$ does not. Then we have
\[ \Gal(K_m/K(x_1+x_2,x_1x_2,\z_m))=\langle\tau_2\rangle \]
and $[K(x_1+x_2,x_1x_2,\z_m):K(x_1+x_2,x_1x_2)]=2$.
If $y_1^2\in K(x_1+x_2,x_1x_2,\z_m)$, then $y_1^2=\tau_2(y_1)^2=y_2^2$, giving $y_1=\pm y_2$ and we already ruled this out.
%This would imply that $K_m=K(x_1,x_2,y_1)$ has degree $2$ over $K(x_1,x_2)$, contradicting the assumptions of this case.
Then the degree of the extensions
\[ K(x_1+x_2,x_1x_2) \subset K(x_1+x_2,x_1x_2,\z_m)\subset K(x_1+x_2,x_1x_2,\z_m,y_1) \]
are, respectively, $2$ and at least $4$. Since the extension $K_m/K(x_1+x_2,x_1x_2)$ has degree $8$
the statement follows.

\noindent\emph{The case $[K_m:K(x_1,x_2)]=2$.} Since $|H|=2$ and $|G/H|=2$, then $|G|=4$.
We have to exclude $G=\langle\tau_2\tau_1, -\Id\rangle$,
because these automorphisms fix both $x_1$ and $x_2$, so we would have $G= H$. We are left with $H=\langle -\Id\rangle$
and one the following two possibilities:
\[ G=\langle\tau_2\rangle\qquad{\rm or}\qquad G=\langle\tau_1,-\Id\rangle \ .\]
We now consider each of the two subcases separately.
Assume $G=\langle\tau_2\rangle$ and recall that $y_1\neq\pm y_2$.
%indeed, if $y_1=\pm y_2\,$, then $\tau_2$ fixes
%$y_1$ or $y_2$ and so  $y_1,y_2$ are both contained in $K(x_1+x_2,x_1x_2)$. This yields $K_m=K(x_1,x_2)$, a contradiction.
%So we know  $y_1\neq \pm y_2\,$.
Then $y_1$ and $y_1^2$ are not fixed by any element in $G$, i.e.,
\[ [K(x_1+x_2,x_1x_2, y_1):K(x_1+x_2,x_1x_2)]=4\]
and $K(x_1+x_2,x_1x_2, y_1)=K_m\,$.
Now assume $G=\langle\tau_1,-\Id\rangle$: since $\tau_1$ does not fix $\z_m$ while $-\Id$ does, we have
\[ K(x_1,x_2)=K(x_1+x_2,x_1x_2,\z_m)\ .\]
Hence $K(x_1+x_2,x_1x_2,\z_m,y_1)=K(x_1,x_2,\z_m,y_1)=K_m\,$.
\end{proof}

\begin{remark}\label{Edihedral} The equality $K_2=K(x_1+x_2,x_1x_2,\z_2,y_1)$ does not hold in general.
Indeed it is equivalent to $K_2=K(x_1+x_2,x_1x_2)$ and one can take $\E:\,y^2=x^3-1$ (defined over $\Q$)
and the points $\{P_1=(\z_3, 0), P_2=(\z_3^2, 0)\}$ (as a $\Z$-basis for $\E[2]\,$) to get
$K_2=\Q({\boldsymbol\mu}_3)$ and $\Q(x_1+x_2,x_1x_2)=\Q$. The equality would hold for any other basis, but
the previous theorems allow total freedom in the choice of $P_1$ and $P_2\,$.
\end{remark}

\section{The equality $K_m=K(x_1,\z_m,y_2)$}\label{MinSet}
We start by proving the equality  $K_m=K(x_1,\z_m,y_1,y_2)$ for every odd $m\geqslant 5$.
The cases $m=2$, 3 and 4 are treated in Remark \ref{Eordinates1}, Section \ref{Secm=3} and Section \ref{Secm=4}
respectively.

\begin{theorem} \label{ordinates1}
Let $m\geqslant 4$. If $m$ is an odd number, then $K_m=K(x_1,\z_m,y_1,y_2)$. If $m$ is an even number, then $K_m$ is larger than
$K(x_1,\z_m,y_1,y_2)$ if and only if $[K_m:K(x_1,\z_m,y_1,y_2)]=2$ and its Galois group is generated by the element sending
$P_2$ to $\frac{m}{2}P_1+P_2$. In particular, if $m$ is even then $K_{\frac{m}{2}}\subseteq K(x_1,\z_m,y_1,y_2)$.
\end{theorem}

\begin{proof}
Let $\s\in \Gal(K_m/K(x_1,\z_m,y_1,y_2))$ and write $\s(P_2)=\alpha P_1 + \beta P_2\,$ for some integers
$0\leqslant \alpha, \beta\leqslant m-1$. Since $P_1$ and $\z_m$ are $\s$-invariant we get
\[ \z_m = \s(\z_m) = \s(e_m(P_1,P_2)) = \z_m^\beta \ ,\]
yielding $\beta=1$ and $\s(P_2)=\alpha P_1+P_2\,$. Since $K_m=K(x_1,\z_m,y_1,y_2,x_2)$ and
$x_2$ is a root of $X^3+AX+B-y_2^2\,$, the order of $\s$ is at most $3$. Assume now that $\s\neq \Id$.\\
\emph{If the order of $\s$ is $3$}: we have
\[ P_2=\s^3(P_2)=3\alpha P_1+P_2 \]
hence $3\alpha \equiv 0\pmod{m}$. Moreover, the three distinct points $P_2\,$, $\s(P_2)$ and $\s^2(P_2)$ are on the line $y=y_2\,$.
Thus their sum is zero, i.e.,
\[ O=P_2+\s(P_2)+\s^2(P_2)=3\alpha P_1 +3P_2 \ .\]
Since $3\alpha \equiv 0\pmod{m}$, we deduce $3P_2=O$, contradicting $m\geqslant 4$.\\
\emph{If the order of $\s$ is $2$}: as above $P_2=\s^2(P_2)$ yields $2\alpha \equiv 0 \pmod{m}$. If $m$ is odd this
implies $\alpha\equiv 0 \pmod{m}$, i.e., $\s$ is the identity on $\mathcal E[m]$, a contradiction. If $m$ is even the only
possibility is $\alpha=\frac{m}{2}$.\\
The last statement for $m$ even follows from the fact that $\s$ acts trivially on $2P_1$ and $2P_2\,$.
\end{proof}

\begin{corollary}
Let $p\geqslant 5$ be prime, then $[K_p:K(\z_p,y_1,y_2)]$ is odd.
\end{corollary}

\begin{proof}
Assume there is a $\s\in \Gal(K_p/K(\z_p,y_1,y_2))$ of order $2$. For $i\in\{1, 2\}$,
since $y_i\neq 0$ (because $p\neq 2$), one has $\s(P_i)\neq -P_i$ and $\s(P_i)+P_i$ is a nontrivial $p$-torsion point
lying on the line $y=-y_i\,$. If $\s(P_i)+P_i$ is not a multiple of $P_j$ ($i\neq j$); then the set
$\{P_j,\s(P_i)+P_i\}$ is a basis of $\E[p]$. Let $\s(P_i)+P_i=:(\tilde{x}_i, -y_i)$; then by Theorem \ref{ordinates1},
we have $K(\z_p, \tilde{x}_i,y_1,y_2)=K_p\,$.
But $\s$ acts trivially on $\z_p\,$, $y_1$ and $y_2$ by definition and on $\tilde{x}_i$ as well (because
$\s(\s(P_i)+P_i)=P_i+\s(P_i)\,$). Hence $\s$ fixes $K_p$ which contradicts $\s\neq\Id$.

\noindent Therefore $\s(P_1)=-P_1+\b_1P_2$ and $\s(P_2)=\b_2P_1-P_2$ which, together with $\s^2=\Id$, yield $\b_1=\b_2=0$.
Hence both $P_1$ and $P_2$ are mapped to their opposite: a contradiction to $\s(y_i)=y_i\,$.
\end{proof}

\begin{remark}\label{Eordinates1}
The equality $K_2=K(x_1,\z_2,y_1,y_2)$ does not hold in general. A counterexample is again provided by the curve
$\E:\,y^2=x^3-1$ with $P_1=(1, 0)$ (as in Remark \ref{Edihedral} any other choice would yield the equality $K_2=K(x_1)\,$).
\end{remark}

\noindent Before going to the main theorem we show a little application for primes $p\equiv 2\pmod 3$.

\begin{theorem} \label{ordinates1+}
Let $p\equiv 2\pmod 3$ be an odd prime, then $K_p=K(x_1,y_1,y_2)$ or $K_p=K(x_1,y_1,\z_p)$.
\end{theorem}

\begin{proof}
The degree of $x_2$ over $K(y_2)$ is at most $3$, hence $[K_p:K(x_1,y_1,y_2)]\leqslant 3$.
By Theorem \ref{ordinates1} we have the equality $K_p=K(x_1,\z_p,y_1,y_2)$ and the hypothesis ensures that $[\Q(\z_p):\Q]$
is not divisible by $3$, so the same holds for $[K_p:K(x_1,y_1,y_2)]$.
Thus either $K_p=K(x_1,y_1,y_2)$ or $[K_p:K(x_1,y_1,y_2)]=2$. If the second case occurs, then take the nontrivial element $\s$
of $\Gal(K_p/K(x_1,y_1,y_2))$. Since $\s$ fixes $x_1,\,y_1$ and $y_2$, it can be written as
\[ \s=\left( \begin{array}{cc} 1 & b \\ 0 & d \\ \end{array} \right)\quad {\rm with} \quad
\s^2=\left( \begin{array}{cc} 1 & b(1+d) \\ 0 & d^2 \\ \end{array} \right) \ .\]
Since $p$ is an odd prime, then $\s^2=\Id$ leads either to $d=1$ (hence $b=0$ and $\s=\Id$, a contradiction) or to $d=-1$.
Hence $\s(P_2)=bP_1-P_2$ (with $b\neq 0$ otherwise $\s$ would fix $x_2$ as well), i.e.,
$bP_1$ lies on the line $y=-y_2\,$. Thus $K(y_2)\subseteq K(x_1,y_1)$ and so $K_p=K(x_1,y_1,\z_p)$.
\end{proof}

\begin{corollary} \label{monogenic}
Let $p\equiv 2\pmod 3$ be an odd prime. Assume that $\E$ has a $K$-rational torsion point $P_1$ of order $p$.
Then either $K_p=K(\z_p)$ or $K_p=K(y_2)$.
\end{corollary}

%\begin{proof}
%Just apply Theorem \ref{ordinates1+} to a basis $\{P_1,\,P_2\}$.
%\end{proof}

%\begin{remark} In \cite{Mer} Merel proved that if $\Q(\E[p])=\Q(\z_p)$, then $p\in \{2,3,5\}$. Thus, in the
%hypotheses of Corollary \ref{monogenic}, if $K=\Q$ and $p>5$, then $K_p=K(y_2)$.
%\end{remark}

\noindent We are now ready to prove the equality appearing in the title of this section.

\begin{theorem}\label{ordinates1++}
If $m\geqslant 4$ and $K_m=K(x_1,\z_m,y_1,y_2)$, then $K_m=K(x_1,\z_m,y_2)$ (in particular this holds
for any odd $m\geq 5$, by Theorem \ref{ordinates1}).
\end{theorem}

\begin{proof}
The hypotheses imply $K_m=K(x_1,\z_m,y_2)(y_1)$ so $[K_m:K(x_1,\z_m,y_2)]\leqslant 2$.
Take $\s\in \Gal(K_m/K(x_1,\z_m,y_2))$, then $\s(x_1)=x_1$ yields $\s(P_1)=\pm P_1\,$. If $\s(P_1)=P_1\,$,
then $y_1\in K(x_1,\z_m,y_2)$ and $K_m=K(x_1,\z_m,y_2)$. Assume that $\s(P_1)=-P_1$ and let
\[ \s=\left(\begin{array}{cc} -1 & a \\ 0 & b \end{array} \right) \ .\]
Using the Weil pairing (recall $\z_m:=e_m(P_1,P_2)\,$), we have
$\z_m= \s(\z_m)= \z_m^{-b}\,$,
which yields $b\equiv -1 \pmod{m}$, while
\[ \s^2=\left(\begin{array}{cc} 1 & -2a \\ 0 & 1 \end{array} \right)\,=\,\Id \]
leads to $\-2a \equiv 0\pmod{m}$.

\noindent {\em Case $a\equiv 0 \pmod{m}$}: we have $\s=-\Id$. Then $\s(P_2)=-P_2\,$, i.e.,
$\s(x_2)=x_2\in K(x_1,\z_m,y_2)$. By Theorem \ref{zetaversusy}, this yields $K_m= K(x_1,\z_m,y_2)$
and contradicts $\s\neq \Id$.
%(alternative: $\s(P_2)=-P_2$ yields $y_2=\s(y_2)=-y_2\,$, i.e., $y_2=0$ so $P_2$ has order 2
%which contradicts $m\geqslant 4$).

\noindent {\em Case $a\equiv \frac{m}{2} \pmod{m}$}: we have $\s(P_2)=\frac{m}{2}P_1-P_2\,$,
i.e., $\s(P_2)+P_2-\frac{m}{2}P_1=O$. Since $P_2$ and $\s(P_2)$ lie on the line $y=y_2$ and are distinct, then $-\frac{m}{2}P_1$
must be the third point of $\E$ on that line. Since $-\frac{m}{2}P_1$ has order 2 this yields $y_2=0$,
contradicting $m\geqslant 4$.
\end{proof}

To provide generators for a more general $m$ one can also use the following lemma.

\begin{lemma}\label{LemRey}\ \begin{itemize}
\item[{\bf 1.}] Assume that $P \in E(K)$ is not a $2$-torsion point and that $\phi: E \to E$ is a $K$-rational isogeny with $\phi(R)=P$.
Then $K(x(R), y(R))=K(x(R))$.
\item[{\bf 2.}] If $R$ is a point in $\E(\ov{K})$ and $n\geqslant 1$, then we have $x(nR)\in K(x(R))$.
\end{itemize} \end{lemma}

\begin{proof}
Part {\bf 1} is \cite[Lemma 2.2]{Re} and part {\bf 2} is well known.
%Put $F=K(x(R))$ and $F'=K(x(R), y(R))$, then $[F':F] \leqslant 2$ and we take $\s\in\Gal(F'/F)$.
%Since $\s$ fixes $x(R)$, one has $\s(R)=\pm R$. Moreover $\s(\phi(R))-\phi(R)= O$ yields $\phi(\s(R)-R)=O$ as well.
%Now $\s(R)=-R$ would yield $O=\phi(-2R)=-2P$ a contradiction to $P\not\in \E[2]$. Hence $\s(R)=R$ and $F'=F$.
\end{proof}

%\begin{lemma}\label{contenimento-x}
%If $R$ is a point in $\E(\ov{K})$ and $n\geqslant 1$, then we have $x(nR)\in K(x(R))$.
%\end{lemma}

%\begin{proof}
%For any $\s\in \Gal(\ov{K}/K(x(P)))$ one has $\s(P)=\pm P$ and $\s(nP)=\pm nP$.
%Hence $\s(x(nP))=x(nP)$, i.e., $x(nP)\in K(x(P))$.
%\end{proof}

\begin{proposition}\label{corReynolds}
Let $m$ be divisible by $d\geqslant 3$ and let $R$ be a point of order $m$. Then
\[ K(x(R), y(R))=K\left(x(R), y\left(\frac{m}{d}R\right)\,\right) \ .\]
In particular, if $K=K(\E[d])$ and $R$ is a point of order $m$, then $K(x(R), y(R))=K(x(R))$.
\end{proposition}

\begin{proof}
Apply the previous lemma to the field $K(P)$, with $P=\frac{m}{d}R$ and $\phi=\left[\frac{m}{d}\right]$.
\end{proof}

\begin{corollary}
Let $m$ be divisible by an odd number $d\geqslant 5$. Then
\[ K_m = K\left(x(P_1),x(P_2),\z_d,y\left(\frac{m}{d}P_2\right)\,\right) \ .\]
\end{corollary}

\begin{proof}
By Proposition \ref{corReynolds}, $K_m=K_d(x(P_1),x(P_2))$. Obviously $\left\{\frac{m}{d}P_1,\frac{m}{d}P_2\right\}$ is
a $\Z$-basis for $\E[d]$, hence Theorem \ref{ordinates1} and Theorem \ref{ordinates1++} (applied with $m=d$) yield
\[ K_d=K\left(x\left(\frac{m}{d}P_1\right),\z_d,y\left(\frac{m}{d}P_2\right)\,\right) \ .\]
By Lemma \ref{LemRey}, we have $x\left(\frac{m}{d}P_1\right)\in K(x(P_1))$ and the corollary follows.
\end{proof}

The previous result leaves out only integers $m$ of the type $2^s3^t$. For the case $t=1$
we mention the following

\begin{proposition}\label{computing}
Assume ${\rm char}(K)\neq 2,3$, then the coordinates of the points of order dividing $3\cdot 2^n$
can be explicitly computed by radicals out of the coefficients of the Weierstrass equation.
\end{proposition}

\begin{proof}
By the Weierstrass equation, we can compute the $y$-coordinates out of the $x$-coordinate. Then by the addition formula,
it suffices to compute the $x$-coordinate of two $\Z$-independent points of order $3$ (done in Section \ref{Secm=3}),
and the $x$-coordinate of two $\Z$-independent points of order $2^n$ (done in Section \ref{Secm=4} for $n=1,2$).
The coordinate $x(P)$ of a point $P$ of order $2^n$ (with $n\geqslant 3$)
can be computed from $x(2P)$. Indeed, we have $y(P)\neq 0$ (because the order of $P$ is not $2$)
and so, by the duplication formula,
\[ x(2P)=\frac{x(P)^4-2Ax(P)^2-8Bx(P)+A^2}{4x(P)^3+4Ax(P)+4B } \]
(a polynomial equation of degree $4$ with coefficients coming from the Weierstrass equation).
\end{proof}

\begin{proposition}
If $m$ is divisible by $3$ (resp. $4$), then
\[ K_{m}=K_{m,x}\cdot K(y(Q_1),y(Q_2)) \]
where $\{Q_1\,,Q_2\}$ is a  $\Z$-basis for $\E[3]$ (resp. $\E[4]$).
\end{proposition}

\begin{proof}
Just apply Proposition \ref{corReynolds} with $d=3$ (resp. $d=4$).
\end{proof}

\section{Galois representations and exceptional primes}\label{SecGal}
\noindent We begin with some remarks on the Galois group $\Gal(K_p/K)$ for a prime $p\geqslant 5$,
which led us to believe that the generating set $\{x_1,\z_p,y_2\}$ is often minimal.

\begin{lemma}\label{ExtOrd2p}
For any prime $p\geqslant 5$ one has $[K_p:K(x_1,\z_p)] \leqslant 2p$.
Moreover the Galois group $\Gal(K_p/K(x_1,\z_p))$ is cyclic, generated
by a power of $\eta=\left(\begin{array}{cc} -1 & 1 \\ 0 & -1 \end{array}\right)\,$.
\end{lemma}

\begin{proof} By Theorem \ref{ordinates1++}, we have $K_p=K(x_1,\z_p,y_2)$. Let $\s\in \Gal(K_p/K(x_1,\z_p))$, then
$\s(P_1)=\pm P_1$ and $\det(\s)=1$ yield $\s=\left(\begin{array}{cc} \pm 1 & \alpha \\ 0 & \pm 1 \end{array}\right)$
(for some $0\leqslant \alpha \leqslant p-1$).
The powers of $\eta$ are
\[ \eta^n = \left\{ \begin{array}{ll} \left(\begin{array}{cc} 1 & -n \\ 0 & 1 \end{array}\right) & {\rm if}\ n\ {\rm is\ even}\\
\ & \ \\
\left(\begin{array}{cc} -1 & n \\ 0 & -1 \end{array}\right) & {\rm if}\ n\ {\rm is\ odd} \end{array} \right. \]
and its order is obviously $2p\,$; clearly any such $\s$ is a power of $\eta$.
%One easily checks that, for even $\alpha$,
%\[ \left(\begin{array}{cc} -1 & \alpha \\ 0 & -1 \end{array}\right) = \eta^{\alpha+p} \quad {\rm and} \quad
%\left(\begin{array}{cc} 1 & \alpha \\ 0 & 1 \end{array}\right) = \eta^{(p-1)\alpha} \]
%while, for odd $\alpha$,
%\[ \left(\begin{array}{cc} -1 & \alpha \\ 0 & -1 \end{array}\right) = \eta^\alpha \quad {\rm and} \quad
%\left(\begin{array}{cc} 1 & \alpha \\ 0 & 1 \end{array}\right) = \eta^{p-\alpha} \ . \]
\end{proof}

\begin{remark}\label{NotNorm}
The group generated by $\eta$ in $\GL_2(\Z/p\Z)$ is not normal; hence, in general, the extension $K(x_1,\z_p)/K$
is not Galois.
\end{remark}

\noindent Since the $p$-th division polynomial has degree $\frac{p^2-1}{2}$ and, obviously, $[K(x_1,\z_p):K(x_1)]\leqslant p-1$
one immediately finds
\[ [K(x_1,\z_p,y_2):K]\leqslant \frac{p^2-1}{2}\cdot(p-1)\cdot 2p  = |\GL_2(\Z/p\Z)| \]
and can provide conditions for the equality to hold.

\begin{theorem}\label{GalGL}
Let $p\geqslant 5$ be  a prime, then $\Gal(K_p/K) \simeq \GL_2(\Z/p\Z)$ if and only if the following hold:\begin{itemize}
\item[{\bf 1.}] $\z_p\not\in K$;
\item[{\bf 2.}] the $p$-th division polynomial $\varphi_p$ is irreducible in $K(\z_p)[x]$;
\item[{\bf 3.}] $y_1\not\in K(\z_p,x_1)$ and the generator of $\Gal(K(\z_p,x_1,y_1)/K(\z_p,x_1))$ is not $-\Id$. %(in $\Gal(K_p/K)\,$).
\end{itemize}
\end{theorem}

\begin{proof} Let $\s$  be a generator of \hspace{0.02cm} $\Gal(K(\z_p,x_1,y_1)/K(\z_p,x_1))$. Then $\s(P_1)=-P_1$
(because of hypothesis {\bf 3}) and $\det(\s)=1$. Hence it is of type
$\s= \left(\begin{array}{cc} -1 & \alpha \\ 0 & -1 \end{array}\right)$ with $\alpha\neq 0$ (again
by hypothesis {\bf 3}). Therefore $\s$ has order $2p$ in $\Gal(K_p/K(\z_p,x_1))$ and the hypotheses lead to the equality
$[K_p:K]=|\GL_2(\Z/p\Z)|\,$. Vice versa it is obvious that if any of the conditions does not hold we get
$[K_p:K]<|\GL_2(\Z/p\Z)|\,$.\end{proof}

\begin{remark}\label{Serre}
As mentioned in the Introduction, if $K$ is a number field and $\E$ has no complex multiplication, then one expects
the equality to hold for almost all primes $p$ (for a recent bound on exceptional primes for which $\rho_{\E,p}$ is
not surjective see \cite{LV}). Hence for a general number field $K$ (which, of course, can contain $\z_p$ or some coordinates
of generators of $\E[p]$ only for finitely many $p$) one expects $\{x_1,\z_p,y_2\}$ to be a minimal set of generators
for $K_p$ over $K$ (among those contained in $\{x_1,x_2,y_1,y_2,\z_p\}\,$).
We have encountered an exceptional case in Theorem \ref{ordinates1+}, where for $p\equiv 2\pmod 3$ ($p\neq 2$) one could have
$K_p=K(x_1,y_1,\z_p)$. If this is the case, the maximum degree for $[K_p:K]$ is $\frac{p^2-1}{2}\cdot 2\cdot (p-1)$.
Therefore for infinitely many primes $p\equiv 2\pmod 3$ we have
$K_p=K(x_1,y_1,y_2)=K(x_1,\z_p,y_2)\neq K(x_1,y_1,\z_p)$
(which emphasizes the need for coordinates of $P_2$ in our generating set).
\end{remark}

\begin{definition}\label{ExcPrimes}
For an elliptic curve $\E$ defined over a number field $K$ and a prime $p$ we say that $p$ is {\em exceptional} for
$\E$ if $\rho_{\E,p}$ is not surjective, i.e., if $[K_p:K]< |\GL_2(\Z/p\Z)|$. In particular,
if $\E$ has complex multiplication, then all primes are exceptional for $\E$, because $K_p/K$ is an abelian extension
(see, e.g., \cite[Chapter II, \S 5]{Sil2}).
\end{definition}

\noindent In the rest of this Section \ref{SecGal} we will investigate the case of exceptional primes,
assuming that $K$ is a number field. For exceptional primes the Galois group $\Gal(K_p/K)$ is a proper subgroup of
$\GL_2(\Z/p\Z)$. Hence it falls in one of the
following cases (see \cite[Section 2]{Ser} for a complete proof or \cite[Lemma 4]{LV} for a similar statement).

\begin{lemma}\label{SubGL2}
Let $G$ be a subgroup of $\GL_2(\Z/p\Z)$ then one of the following holds:
\begin{itemize}
\item[{\bf 1.}] $G$ is contained in a Borel subgroup;
\item[{\bf 2.}] $G$ is a Cartan subgroup;
\item[{\bf 3.}] $G$ is contained in the normalizer of a Cartan subgroup, but it is not a Cartan subgroup;
\item[{\bf 4.}]  the image of $G$ under the projection $\pi: \GL_2(\Z/p\Z) \rightarrow {\rm PGL}_2(\Z/p\Z)$ is
contained in a subgroup which is isomorphic to one of the alternating groups $A_4$ and $A_5$ or to the symmetric group $S_4$.
\end{itemize}
\end{lemma}

\noindent In particular if one of cases \textbf{1} or \textbf{2} holds, then $G$ acts reducibly on $\E[p]$.
Regarding case ${\bf 4}$ we have the next statement.

\begin{lemma}\label{LVLemma}
If $p\geq 53$ is unramified in $K/\Q$ and exceptional for $\E$, then $\Gal(K_p/K)$ does not satisfy
{\bf 4} of Lemma \ref{SubGL2}.
\end{lemma}

\begin{proof}
See \cite[Lemma 8]{LV}, depending on \cite[Lemma 18]{Ser2}.
\end{proof}

\noindent We shall provide some information on the generating sets for $K_p$ when $p$ is exceptional for $\E$
and $\Gal(K_p/K)$ falls in cases {\bf 1}, {\bf 2} or {\bf 3} of Lemma \ref{SubGL2}. We start with the already
mentioned exceptional case appearing in Theorem \ref{ordinates1+} and recall that we are always assuming $p\geq 5$.

\begin{proposition}
If $K_p=K(x_1,y_1,\z_p)$, then $[K_p:K]< (p^2-1)(p-1)$ unless $p=5$ and $\pi(\Gal(K_p/K))\simeq S_4\,$.
\end{proposition}

\begin{proof} We have already noticed that $[K_p:K]\leq (p^2-1)(p-1)$, so the prime $p$ is exceptional.
But the order of a Borel subgroup is $p(p-1)^2$ and the order of a Cartan subgroup is at most $(p-1)^2$
(and it has index 2 in its normalizer), so the statement holds (even with the stronger bound $p(p-1)^2\,$)
when $\Gal(K_p/K)$ falls in cases {\bf 1}, {\bf 2} or {\bf 3} of Lemma \ref{SubGL2}.
Assume we are in case {\bf 4} and note that if $|\Gal(K_p/K)|=(p^2-1)(p-1)$, then $|\pi(\Gal(K_p/K))|\geq p^2-1$.
Thus case {\bf 4} cannot happen for $p\geq 11$. Moreover, if $p=7$, then $p^2-1 > |S_4|$ and ${\rm PGL}_2(\Z/p\Z)$
does not contain $|A_5|$ (see \cite[Section 2.5]{Ser}). We are left with $p=5$, $[K_5:K]=96$ and
$|\pi(\Gal(K_p/K))|\geq 24 = |S_4|$, which completes the proof.
\end{proof}

\subsection{Exceptional primes I: Borel subgroup}
Assume that $p\geq 5$ is exceptional for $\E$ and $\Gal(K_p/K)$ is contained in a Borel subgroup.
We can write elements of $\Gal(K_p/K)$ as upper triangular matrices $\s=\left( \begin{array}{cc} a & b \\ 0 & c \end{array}\right)$
with $ac\neq 0$ (this is not restrictive, since the results of the previous sections were completely independent
of the chosen basis $\{P_1,P_2\}\,$).

\begin{theorem}\label{Borel} Let $p\geq 5$ and assume that $\Gal(K_p/K)$ is contained in a Borel subgroup. \begin{itemize}
\item[{\bf 1.}] If $p\not\equiv 1 \pmod{3}$, then $K_p=K(\z_p,y_2)\,$;
\item[{\bf 2.}] if $p\equiv 1 \pmod{3}$, then $[K_p:K(\z_p,y_2)]$ is $1$ or $3$.
\end{itemize}
\end{theorem}

\begin{proof}
We know $K_p=K(x_1,\z_p,y_2)$. Take an element $\s\in \Gal(K_p/K(\z_p,y_2))$ so that
$\s=\left( \begin{array}{cc} a^{-1} & b \\ 0 & a \end{array}\right)$. Let $P_2\,$, $R_2$ and $S_2$ be the
three points of the curve $\E$ on the line $y=y_2\,$, so that $P_2+R_2+S_2=O$. We have that $\s(P_2)=bP_1+aP_2$ must
be $P_2$ or $R_2$ or $S_2$ (the cases $R_2$ and $S_2$ are obviously symmetric).

\noindent{\em Case 1: $\s(P_2)=P_2\,$.} Then $b=0$, $a=1$ and $\s=\Id$.

\noindent{\em Case 2: $\s(P_2)=R_2\,$.} Then $\s^2(P_2)=a^{-1}bP_1+abP_1+a^2P_2\,$.
\begin{itemize}
\item If $\s^2(P_2)=P_2\,$, then $a^2= 1$ and $a+a^{-1}\neq 0$ yields $b=0$. Hence $\s(P_1)=\pm P_1$ and
$\s$ fixes $x_1\,$. Since $K_p=K(x_1,\z_p,y_2)$, this implies $\s=\Id$.
\item If $\s^2(P_2)=R_2\,$, then one gets $a^2=a$ (i.e., $a=1$) and $2b=b$ (i.e., $b=0$), leading to $\s=\Id$.
\item If $\s^2(P_2)=S_2\,$, then $P_2+R_2+S_2=O$ yields
\[ P_2+bP_1+aP_2+a^{-1}bP_1+abP_1+a^2P_2 = ba^{-1}(a+1+a^2)P_1+(1+a+a^2)P_2=O  \]
Thus $1+a+a^2=0$ and this is possible if and only if $p\equiv 1\pmod{3}$.
\end{itemize}
Therefore, if $p\not\equiv 1\pmod{3}$, we have $\s=\Id$ and $K_p=K(\z_p,y_2)$. If $p\equiv 1\pmod{3}$ and
$1+a+a^2=0$, then the above $\s$ has order 3 and the proof is complete.
\end{proof}

\subsection{Exceptional primes II: Cartan subgroup}
Assume that $p\geq 5$ is exceptional for $\E$ and $\Gal(K_p/K)$ is contained in a Cartan subgroup (resp. in a normalizer of
a Cartan subgroup). Then we can write elements of $\Gal(K_p/K)$ as matrices $\s=\left( \begin{array}{cc} a & 0 \\ 0 & c \end{array}\right)$
(resp. $\s=\left( \begin{array}{cc} a & 0 \\ 0 & c \end{array}\right)$ or
$\s=\left( \begin{array}{cc} 0 & a \\ c & 0 \end{array}\right)\,$) with $ac\neq 0$.

\begin{theorem}\label{Cartan} In the above setting we have $K_p=K(x_1,\z_p)$ or $K(x_1,y_1,\z_p)$. Moreover\begin{itemize}
\item[{\bf 1.}] if $p\not\equiv 1 \pmod{3}$, then $K_p=K(\z_p,y_2)\,$;
\item[{\bf 2.}] if $p\equiv 1\pmod{3}$, then $[K_p:K(\z_p,y_2)]$ is $1$ or $3$.
\end{itemize}
\end{theorem}

\begin{proof} Note that the only elements of the normalizer of a Cartan subgroup (hence, in particular, of a Cartan subgroup)
which fix $x_1$ and $\z_p$ are $\pm \Id\,$: the first statement follows immediately. Now consider $\s\in \Gal(K_p/K(\z_p,y_2))$ and let
 $R_2$ and $S_2$ be the points defined in Theorem \ref{Borel}. If
$\s=\left( \begin{array}{cc} 0 & a \\ -a^{-1} & 0 \end{array}\right)$, then
$\s^2(P_2)=\s(aP_1)=-P_2\,$. Since $\s$ fixes $y_2\,$, this implies $y_2=0$ which contradicts $p\neq 2$. Therefore
we can restrict to Cartan subgroups and consider only $\s=\left( \begin{array}{cc} a^{-1} & 0 \\ 0 & a \end{array}\right)$.

\noindent{\em Case 1: $\s(P_2)=P_2\,$.} Then $a=1$ and $\s=\Id$.

\noindent{\em Case 2: $\s(P_2)=R_2\,$.} Then $\s^2(P_2)=a^2P_2\,$.
\begin{itemize}
\item If $\s^2(P_2)=P_2\,$, then $a^2= 1$ and $\s(P_1)=\pm P_1\,$. As in Theorem \ref{Borel}, this implies $\s=\Id$.
\item If $\s^2(P_2)=R_2\,$, then $a^2=a$ yields $a=1$ and $\s=\Id$.
\item If $\s^2(P_2)=S_2\,$, then $P_2+R_2+S_2=O$ yields
\[ P_2+aP_2+a^2P_2 = (1+a+a^2)P_2=O \ .\]
Thus $1+a+a^2=0$ and this is possible if and only if $p\equiv 1\pmod{3}$.
\end{itemize}
Therefore, if $p\not\equiv 1\pmod{3}$, we have $\s=\Id$ and $K_p=K(\z_p,y_2)$. If $p\equiv 1\pmod{3}$ and $1+a+a^2=0$, then
$\s$ has order 3.
\end{proof}

\begin{remark}
The information carried by $\z_p$ seems more relevant than that by the coordinate $x_1$ in the
exceptional case. Indeed if one considers a $\s\in \Gal(K_p/K(x_1,y_2))$, there is always room for elements like
$\s=\left( \begin{array}{cc} -1 & 0 \\ 0 & 1 \end{array}\right)$ of order 2. A proof similar to the previous
ones leads to (both in the Borel and the Cartan case)\begin{itemize}
\item[{\bf 1.}] $p\not\equiv 1\pmod{3} \Longrightarrow [K_p:K(x_1,y_2)]$ divides 4;
\item[{\bf 2.}] $p\equiv 1\pmod{3} \Longrightarrow [K_p:K(x_1,y_2)]$ divides 12.
\end{itemize}
\end{remark}

\subsection{Remarks on modular curves}\label{SecMod}

We give just an application of the results of the previous sections to the classical modular curves $X(p)$ and $X_1(p)$,
associated to the action of the congruence subgroups
\[ \Gamma(p) = \left\{ A=\left( \begin{array}{cc} a & b \\ c & d \end{array} \right) \in \SL_2(\Z)\,:\,
A\equiv \left( \begin{array}{cc} 1 & 0 \\ 0 & 1 \end{array} \right) \pmod p \right\} \]
and
\[ \Gamma_1(p) = \left\{ A=\left( \begin{array}{cc} a & b \\ c & d \end{array} \right) \in \SL_2(\Z)\,:\,
A\equiv \left( \begin{array}{cc} 1 & * \\ 0 & 1 \end{array} \right) \pmod p \right\} \]
on the complex upper half plane $\mathcal{H}=\{z\in \mathbb{C}\,:\,Im\,z > 0\}$ via M\"obius trasformations
(for detailed definitions and properties see, e.g. \cite{KM} or \cite{Sh}).
We recall that $X(p)$ and $X_1(p)$ parametrize families of elliptic curves with some extra {\it level $p$ structure}
via their moduli interpretation. Namely \begin{itemize}
\item[$\bullet$] non cuspidal points in $X(p)$ correspond to triples $(\E,P_1,P_2)$ where $\E$ is an elliptic curve
(defined over $\mathbb{C}$) and $P_1$, $P_2$ are points of order $p$ generating the whole group $\E[p]$;
\item[$\bullet$] non cuspidal points in $X_1(p)$ correspond to couples $(\E,Q)$ where $\E$ is an elliptic curve
(defined over $\mathbb{C}$) and $Q$ is a point of order $p$
\end{itemize}
(all these correspondences have to be considered modulo the natural isomorphisms).

\noindent
Let $K$ be a number field. The points of $X(p)$ or $X_1(p)$ which are rational over $K$ will be
denoted by $X(p)(K)$ or $X_1(p)(K)$. Obviously a point is $K$-rational if and only if it is
$\Gal(\overline{\Q}/K)$-invariant (in particular, with the representation provided above one needs an elliptic
curve $\E$ defined over $K$).

\begin{definition}\label{DefExMod}
A point $(\E,P_1,P_2)\in X(p)$ (resp. $(\E,P_1)\in X_1(p)\,$) is said to be {\em exceptional}
if $p$ is exceptional for $\E$. In particular, if $\E$ is defined over $K$, we call such a point {\em Borel exceptional}
(resp. {\em Cartan exceptional}) if $\Gal(K(\E[p])/K)$ is contained in a Borel subgroup
(resp. in the normalizer of a Cartan subgroup).
\end{definition}

\noindent The following is an easy consequence of Theorem \ref{ordinates1++}.

\begin{corollary}
Assume $p\geqslant 5$; let $\E$ be an elliptic curve defined over a number field $K$ and let $P\in \E[p]$ be of order $p$.
For any field $L$ containing $K(x(P),\z_p)$ or containing $K(y(P),\z_p)$ and for any point $Q\in \E[p]$
independent from $P$, we have
\[ (\E,Q) \in X_1(p)(L) \iff (\E,P,Q)\in X(p)(L) \ .\]
\end{corollary}

\begin{proof} The arrow $\Leftarrow$ is obvious. Now assume $(\E,Q)\in X_1(p)(L)$, then
\[ L\supseteq K(x(P),\z_p,y(Q))=K_p\quad {\rm or}\quad L\supseteq K(y(P),\z_p,x(Q))=K_p \]
(both final equalities hold because of Theorem \ref{ordinates1++}). Hence $(\E,P,Q)\in X(p)(L)$.
\end{proof}

\noindent
It would be interesting to describe the families of elliptic curves for which the previous corollary
becomes trivial, i.e., curves for which $K(x(P),\z_p)$ or $K(y(P),\z_p)$ contain $K(x(P),y(P))$. Some examples
are provided by the exceptional primes $\equiv 1\pmod{3}$ for which $K(\z_p,y(P))=K_p$.

\noindent On exceptional points we have the following

\begin{corollary}
Assume $p\geq 53$ is unramified in $K/\Q$ and $p\not\equiv 1 \pmod{3}$, then, for any field $L\supseteq K(\z_p)$,
the $L$-rational exceptional points of $X(p)$ and $X_1(p)$ are associated to the same elliptic curves. The same statement
holds for $p\equiv 1\pmod{3}$ as well if we restrict to Cartan exceptional points.
\end{corollary}

\begin{proof} We only need to check that if $(\E,Q)\in X_1(p)(L)$ is exceptional, then  $(\E,Q,R)\in X(p)(L)$,
for any $R$ completing $Q$ to a $\Z$-basis of $\E[p]$. For $p\not\equiv 1\pmod{3}$, this immediately follows from
\[ L\supseteq K(\z_p,y(Q)) = K_p, \]
by Theorems \ref{Borel} and \ref{Cartan}. If $p\equiv 1\pmod{3}$ (and $(\E,Q)$ is Cartan exceptional), then
Theorem \ref{Cartan} shows that
\[ L\supseteq K(\z_p,x(P),y(Q)) = K_p \ .\qedhere\]
\end{proof}

\section{Fields $K(\E[3])$}\label{Secm=3}
In this section we generalize the classification of the number fields $\Q(\E[3])$, appearing in \cite{BP},
to the case when the characteristic of the base field $K$ is different from 2 and 3.
Under the last assumption on $K$ we have that $\E$ can be written in Weierstrass form $y^2=x^3+Ax+B$.
We recall that the four $x$-coordinates of the $3$-torsion points of $\E$
are the roots of the polynomial $\varphi_3:=x^4+2Ax^2+4Bx-A^2/3$.
Solving $\varphi_3$ with radicals, we get explicit expressions for the $x$-coordinates and we recall that
for $m=3$ being $\Z$-independent is equivalent to having different $x$-coordinates. Let $\D:=-432B^2-64A^3$
be the discriminant of the elliptic curve.
If $B\neq 0$, the roots of $\varphi_3$ are
\[ x_1=-\frac{1}{2}\sqrt{\frac{\sqrt[3]{\Delta}-8A}{3}-\frac{8B\sqrt{3}}{\sqrt{-\sqrt[3]{\Delta}-4A}}}
+\frac{\sqrt{-\sqrt[3]{\Delta}-4A}}{2\sqrt{3}}\ ,\]
\[ x_2=\frac{1}{2}\sqrt{\frac{\sqrt[3]{\Delta}-8A}{3}-\frac{8B\sqrt{3}}{\sqrt{-\sqrt[3]{\Delta}-4A}}}
+\frac{\sqrt{-\sqrt[3]{\Delta}-4A}}{2\sqrt{3}}\ ,\]
\[ x_3=-\frac{1}{2}\sqrt{\frac{\sqrt[3]{\Delta}-8A}{3}+\frac{8B\sqrt{3}}{\sqrt{-\sqrt[3]{\Delta}-4A}}}
-\frac{\sqrt{-\sqrt[3]{\Delta}-4A}}{2\sqrt{3}}\ ,\]
\[ x_4=\frac{1}{2}\sqrt{\frac{\sqrt[3]{\Delta}-8A}{3}+\frac{8B\sqrt{3}}{\sqrt{-\sqrt[3]{\Delta}-4A}}}
-\frac{\sqrt{-\sqrt[3]{\Delta}-4A}}{2\sqrt{3}}\ .\]
(where we have chosen one square root of $\c$ and one cubic root for $\D$; since $\z_3\in K_3$ the degree $[K_3:K]$
will not depend on this choice).

To ease notation, we define
\[  \c :=\frac{-\sqrt[3]{\D}-4A}{3} \ ,\ \d:= \frac{(-\c-4A)\sqrt{\c}-8B}{\sqrt{\c}} \ \ {\rm and}\ \
\d' := \frac{(-\c-4A)\sqrt{\c}+8B}{\sqrt{\c}}\ .  \]
Thus, when $B\neq 0$, the roots of $\varphi_3$ are
\[  x_1  =\frac{1}{2}(-\sqrt{\d}+\sqrt{\c})\ ,\
x_2=\frac{1}{2}(\sqrt{\d}+\sqrt{\c}) \ ,\
 x_3=\frac{1}{2}(-\sqrt{\d'}-\sqrt{\c}) \ {\rm and } \
 x_4=\frac{1}{2}(\sqrt{\d'}-\sqrt{\c}) \ . \]
The corresponding points $P_i:=(x_i,\sqrt{x_i^3+Ax_i+B})$ have order $3$ and are pairwise $\Z$-in\-de\-pen\-dent
(this would hold with any choice for the sign of the square root providing the $y$-coordinate).
For completeness, we show the expressions of $y_1$, $y_2$, $y_3$ and $y_4$ in terms of $A$, $B$, $\c$, $\d$ and $\d'$:
\[ y_1=\sqrt{\frac{(-\c\sqrt{\c}+4B)\sqrt{\d}+\c\d}{4\sqrt{\c}}} \ \ , \ \
y_2:=\sqrt{\frac{(\c\sqrt{\c}-4B)\sqrt{\d}+\c\d}{4\sqrt{\c}}} \ ,\]
\[ y_3=\sqrt{\frac{(-\c\sqrt{\c}-4B)\sqrt{\d'}-\c\d'}{4\sqrt{\c}}}\ \ , \ \
y_4=\sqrt{\frac{(\c\sqrt{\c}+4B)\sqrt{\d'}-\c\d'}{4\sqrt{\c}}}\ .\]

If $B=0$, then $\c=0$ too and the formulas provided above do not hold anymore.
The $x$-coordinates are now the roots of $\varphi_3=x^4+2Ax^2-A^2/3\,$. Let
\[ \beta:=-\left(\frac{2\sqrt{3}}{3}+1\right)A \ \ {\rm and}\ \  \eta:=\left(\frac{2\sqrt{3}}{3}-1\right)A,\]
then the roots of $\varphi_3$ are $x_1= \sqrt{\beta}$, $x_2=-\sqrt{\beta}$, $x_3=\sqrt{\eta}$ and $x_4=-\sqrt{\eta}$.
Furthermore
\[ y_1=\sqrt{\frac{-2A\sqrt{\beta}}{\sqrt{3}}}=\sqrt{\frac{-2A}{3}\sqrt{-2A\sqrt{3}-3A}} \ .\]

\noindent Using the results of the previous sections and the explicit formulas, we can now give the following
description of $K_3$ in terms of generators.

\begin{proposition} \label{description}
In any case $K_3=K(x_1,y_1,y_2)$. Moreover\begin{itemize}
\item[{\bf 1.}] if $B\neq 0$, then $K_3=K(\sqrt{\c},\z_3,y_1)$;
\item[{\bf 2.}] if $B= 0$, then $K_3=K(\z_3,y_1)$.
\end{itemize}
\end{proposition}

\begin{proof}If $B\neq 0$, then
\[ y_1^2+y_2^2 = -4B-\frac{\c^2}{2\sqrt{\c}}-2A\sqrt{\c} \ .\]
Therefore $x_1+x_2=\sqrt{\c}\in K(y_1^2,y_2^2)$ and $x_2\in K(x_1,y_1^2,y_2^2)$, which immediately
yields $K_3=K(x_1,y_1,y_2)$. Moreover, by Theorem \ref{dihedral}, $K_3=K(x_1+x_2,x_1x_2,\z_3,y_1)$.
So, since
\[ x_1x_2 = \frac{\c}{2}+A+\frac{2B}{\sqrt{\c}} \in K(x_1+x_2)=K(\sqrt{\c})\ ,\]
one has $K_3=K(\sqrt{\c},\z_3,y_1)$.\\
If $B=0$, then $x_1=\sqrt{\beta}=-x_2$ so $K_3=K(x_1,y_1,y_2)$ is obvious.
The final statement follows from $x_1+x_2=0$, $K(x_1x_2)=K(\sqrt{3})\subseteq K(y_1)$ and Theorem \ref{dihedral}.
\end{proof}

\noindent We shall use the statements of Proposition \ref{description} to describe the fields
$K_3$ in terms of the degree $[K_3:K]$ and the Galois groups $\Gal(K_3/K)$.

\subsection{The degree $[K_3:K]$}\label{SecdegK3}
Because of the embedding $$\Gal(K_n/K)\hookrightarrow \GL_2(\Z/n\Z)$$ \noindent one has that
$[K_3:K]$ is a divisor of $|\GL_2(\Z/3\Z)|=48$ (in particular, if $B=0$, then $K_3=K(\z_3,y_1)$ and $y_1$
has degree at most 8 over $K$ so $d:=[K_3:K]$ divides 16). Therefore
$d\in \Omega:=\{1,2,3,4,6,8,12,16,24,48\}.$ In \cite{BP}, we proved that the minimal set for $[\Q(\E[3]):\Q]$ is
$\widetilde{\Omega}:=\{2,4,6,8,12,16,48\}$ and showed also explicit examples for any degree $d\in\widetilde{\Omega}$. When
$K$ is a number field we can get also examples of degree $1,3$ and $24$: it suffices to take the
curves in \cite{BP} with degree $d\in \{2,6,48\}$ and choose $K=\Q(\zeta_3)$ as base field.
%For examples of degree $d=1$, 3 and 24 one simply takes curves from there with $d=2,\ 6$ and $48$ and put $K=\Q(\z_3)$.
In general, once we have a curve $\E$ defined over $\Q$ with $[\Q(\E[3]):\Q]=48$, we produce examples of any degree $d\in\Omega$
by simply considering the same curve over subfields $K$ of $\Q(\E[3])$ (obviously for those $K$ one has $K_3=\Q(\E[3])\,$).

\begin{theorem} \label{classification_m=3}
With notations as above let $d:=[K_3:K]$.
Consider the following conditions for $B\neq 0$
\[ \begin{array}{lll}
{\bf A1.}\ \sqrt[3]{\D}\notin K\,;\hspace{0.15cm} & {\bf B1.}\ \sqrt{\d}\notin K(\sqrt{\c})\,;\hspace{0.15cm}
& {\bf C.}\ \z_3\notin K(\sqrt{\c},y_1)\,;\\
{\bf A2.}\ \sqrt{\c}\notin K(\sqrt[3]{\Delta})\,;\hspace{0.15cm} & {\bf B2.}\ y_1\notin K(\sqrt{\d})\,; &
\end{array} \]
and the corresponding ones for $B=0$
\[ \begin{array}{ll}
{\bf D1.}\ \sqrt{3}\notin K\,;\qquad & {\bf E.}\ \z_3\notin K(y_1)\,; \\
{\bf D2.}\ \sqrt{\beta}\notin K(\sqrt{3})\,;\qquad & \\
{\bf D3.}\ y_1\notin K(\sqrt{\beta})\ . & \\
\end{array} \]
Then the degrees are the following
\begin{center}\begin{tabular}{|c|c|c|c|c|c|}
\hline
$B$ & $d$ & {\em holding conditions}  & $B$ & $d$ & {\em holding conditions}  \\
\hline
$\neq 0$ & {\em 48} & {\bf A1}, {\bf A2}, {\bf B1}, {\bf B2}, {\bf C} & $\neq 0$ & {\em 4} & {\bf A2}, {\bf B1} \\
\hline
$\neq 0$ & {\em 24} & {\bf A1}, {\bf B1}, {\bf B2}, {\bf C} & $\neq 0$ & {\em 4} & {\bf A2}, {\bf B2} \\
\hline
$\neq 0$ & {\em 24} & {\bf A1}, {\bf A2}, {\bf B1}, {\bf B2} & $\neq 0$ & {\em 4} & {\bf B1}, {\bf B2} \\
\hline
$\neq 0$ & {\em 16} & {\bf A2}, {\bf B1}, {\bf B2}, {\bf C} & $\neq 0$ & {\em 3} & {\bf A1} \\
\hline
$\neq 0$ & {\em 12} & {\bf A1}, {\bf A2}, {\bf B1} & $\neq 0$ & {\em 2} & 1 among {\bf A2}, {\bf B1}, {\bf B2} \\
\hline
$\neq 0$ & {\em 12} & {\bf A1}, {\bf A2}, {\bf B2} & $0$ & {\em 16} & {\bf D1}, {\bf D2}, {\bf D3}, {\bf E} \\
\hline
$\neq 0$ & {\em 12} & {\bf A1}, {\bf B1}, {\bf B2} & $0$ & {\em 8} & {\bf D2}, {\bf D3}, {\bf E} \\
\hline
$\neq 0$ & {\em 8} & {\bf B1}, {\bf B2}, {\bf C} & $0$ & {\em 4} & {\bf D1}, {\bf D3} \\
\hline
$\neq 0$ & {\em 8} & {\bf A2}, {\bf B1}, {\bf B2} & $0$ & {\em 4} & {\bf D2}, {\bf D3} \\
\hline
$\neq 0$ & {\em 6} & {\bf A1} and 1 among {\bf A2}, {\bf B1}, {\bf B2}  & $0$ & {\em 2} & {\bf D1} \\
\hline
 & & & $0$ & {\em 2} & {\bf D3} \\
\hline
\end{tabular}\end{center}
\end{theorem}

\begin{proof}
Everything follows from Proposition \ref{description} and the explicit description of the generators of $K_3\,$;
just note that all conditions (except {\bf A1} which provides an extension of degree 3) yield extensions of
degree 2. We remark that not all possible combinations appear in the table because there are certain relations
between the conditions. Indeed, for $B=0$, condition {\bf D2} implies condition {\bf D3}
(since $y_1=\sqrt{\frac{2A}{\sqrt{3}}}\sqrt[4]{\beta}\,$),
while, if {\bf D2} does not hold, then $x_1\in K(\sqrt{3})$ and $x_3=\sqrt{\left(\frac{2\sqrt{3}}{3}-1\right)A}\in K(\sqrt{3})$
as well. Since $x_1x_3=\frac{A\sqrt{-3}}{3}$, this implies that {\bf E} does not hold.\\
In the same way one sees that if {\bf B1} does not hold then $\d$ and $\d'$
are both squares in $K(\sqrt{\c})$. Therefore $x_i\in K(\sqrt{\c})$ for $1\leqslant i\leqslant 4$ and, by (the proof of)
Theorem \ref{zetaversusy}, $\z_3\in K(\sqrt{\c})$ as well, i.e., {\bf C} does not hold. Moreover if {\bf B2} does not hold,
then $y_1^2\,$, which is of the form $u+v\sqrt{\d}$ for some $u,v\in K(\sqrt{\c})$, is a square in $K(\sqrt{\d})$,
hence $y_2^2=u-v\sqrt{\d}$ is a square as well. In this case we have $\sqrt{\c},\sqrt{\d},y_1,y_2\in K(\sqrt{\d})$,
i.e., $K_3=K(x_1,y_1,y_2)=K(\sqrt{\d})$ (in particular {\bf C} does not hold).
\end{proof}

\subsection{Galois groups.}
We now list all possible Galois groups $\Gal(K_3/K)$ via a case by case analysis (one can easily
connect a Galois group to the conditions in Theorem \ref{classification_m=3}, so we do not write down a
summarizing statement here).

\subsubsection{$\mathbf{B= 0}$} The degree $[K_3:K]$ divides 16. Hence $\Gal(K_3/K)$ is a subgroup of the 2-Sylow subgroup
of $\GL_2(\Z/3\Z)$ which is isomorphic to $SD_8$ (the semidihedral group of order 16). If $d=16$, then
$\Gal(K_3/K)\simeq SD_8$ and, by \cite[Theorem 3.1]{BP}, it is generated by the elements
\[  \varphi_{6,1} \left\{ \begin{array}{rcl} y_1 & \mapsto & y_3 \\
\ &\ & \ \\
\sqrt{-3} & \mapsto & -\sqrt{-3} \end{array}\right. \quad{\rm and}\quad
\varphi_{2,1} \left\{ \begin{array}{rcl} y_1 & \mapsto & y_1 \\
\ &\ & \ \\
\sqrt{-3} & \mapsto & -\sqrt{-3} \end{array}\right.  \]
(here and in what follows the notations for the $\varphi_{i,j}$ are taken from \cite[Appendix A]{BP}).

\noindent Obviously if $d=2$, then $\Gal(K_3/K)\simeq \Z/2\Z$ and $d=1$ yields a trivial group. Hence we are left with
$d=4$ and 8.

\noindent \underline{If $d=8$}: then $\sqrt{3}\in K$ but $\sqrt{-3}\notin K$ which yields $i\notin K$.
Letting $\varphi$ be any element of the Galois group, one has $\varphi(y_1^2)=\pm y_1^2\,$, i.e.,
$\varphi(y_1)=\pm y_1\,,\pm iy_1\,$. Then
\[ \Gal(K_3/K) = \langle \varphi_{6,1}^2\,,\varphi_{2,1}:\varphi_{6,1}^8=\varphi_{2,1}^2=\Id\,,\,
\varphi_{2,1}\varphi_{6,1}^2\varphi_{2,1}=\varphi_{6,1}^6 \rangle \simeq D_4 \]
(the dihedral group of order 8).

\noindent \underline{If $d=4$}: then there are two cases\begin{itemize}
\item[{\bf a.}] $\sqrt{3}\notin K$, $\sqrt{\beta},\z_3\in K(\sqrt{3})$ and $y_1\notin K(\sqrt{3})$, or
\item[{\bf b.}] $\sqrt{3}\in K$, $[K(y_1):K]=4$ and $\z_3\in K(y_1)\,$.
\end{itemize}

\noindent In case {\bf a} there are elements sending $\sqrt{3}$ to $-\sqrt{3}$, hence $x_1$ to $x_3$
and $y_1^2$ to $\pm y_3^2\,$. There are no such elements of order 2, so $\Gal(K_3/K)\simeq \Z/4\Z$
and it is generated by $\varphi_{6,1}\varphi_{2,1}$ or $\varphi_{6,1}^3\varphi_{2,1}$
(note that both fix $\z_3$, hence one can also deduce that this case happens if $\z_3$ belongs to $K$
and $i$ does not).

\noindent In case {\bf b} (as in $d=8$) one has $\varphi(y_1^2)=\pm y_1^2\,$. If $\z_3\in K$, then $i\in K$ as well
and the Galois group is $\langle\,\varphi_{6,1}^2\rangle\simeq \Z/4\Z$. If $\z_3\notin K$, then the Galois group must contain
elements moving $i$ and, among them, the ones sending $y_1^2$ to $\pm y_1^2\,$. All such  elements have order 2. Therefore
$\Gal(K_3/K)\simeq \Z/2\Z\times \Z/2\Z$ and the generators are $\{\varphi_{6,1}^4,\varphi_{2,1}\}$ or
$\{\varphi_{6,1}^4,\varphi_{2,1}\varphi_{6,1}^6\}$.

\subsubsection{$\mathbf{B\neq 0}$} The degree is a divisor of 48. Looking at the subgroups of $\GL_2(\Z/3\Z)$ one sees that certain
orders do not leave any choice: indeed $d=1$, 2, 3, 12, 16, 24 and 48 give
$\Gal(K_3/K)  \simeq  \Id$, $ \Z/2\Z$, $ \Z/3\Z$, $D_6\,$, $SD_8\,$,
$\SL_2(\Z/3\Z)$ and $\GL_2(\Z/3\Z)$ respectively. The remaining orders are $d=4$, 6 and 8.

\noindent\underline{If $d=8$}:  then there are two cases \begin{itemize}
\item[{\bf a.}] $K=K(\sqrt{\c})$, $[K(y_1):K]=4$ and $K_3=K(y_1,\z_3)$, or
\item[{\bf b.}] $K=K(\sqrt[3]{\Delta})$ and $K_3=K(\sqrt{\c},y_1)$.
\end{itemize}

\noindent In case {\bf a}, since all elements of the Galois group fix $\sqrt{\c}$, one has $\varphi(\sqrt{\d})=
\pm \sqrt{\d}$, which yields $\varphi(y_1)\in \{\pm y_1\,,\pm y_2\,\}$. Therefore $\varphi$ has order 1, 2 or 4 and,
since $(\Z/2\Z)^3$ is not a subgroup of $\GL_2(\Z/3\Z)$, we have some elements of order 4 (the ones with $\varphi(y_1)=\pm y_2$).
Moreover there is $\s\in\Gal(K_3/K(y_1))$ with $\s(\z_3)=\z_3^2\,$. Note that in this case $x_1\in K(y_1)$ so
$y_2\not\in K(y_1)$ (otherwise $K_3=K(y_1)$ by Proposition \ref{description}, a contradiction to $[K_3:K]=8$), hence
$\s(y_2)=-y_2\,$. Now it is easy to check that $\Gal(K_3/K)=\langle \varphi,\,\s\,:\, \varphi^4=\s^2=\Id\,,\,
\s\varphi\s=\varphi^3\rangle\simeq D_4\,$, with
\[  \varphi \left\{ \begin{array}{rcl} y_1 & \mapsto & y_2 \\
\ &\ & \ \\
\z_3 & \mapsto & \z_3 \end{array}\right. \qquad{\rm and}\qquad
\s \left\{ \begin{array}{rcl} y_1 & \mapsto & y_1 \\
\ &\ & \ \\
\z_3 & \mapsto & \z_3^2 \end{array}\right. \ . \]

\noindent In case {\bf b}, since $\sqrt{\c}$ is no longer fixed, $\varphi(\d)\in \{\d,\d'\,\}$ and therefore
the image of $y_1$ can be any of the other $y_i$'s. Moreover, once $\varphi(\sqrt{\c})$ and $\varphi(\sqrt{\d})$ are
fixed, $\varphi(y_1)=\pm y_i$ yields $\varphi(y_i)=\pm y_1$. So, again, we have no elements of order 8 (and, as above, they
cannot all be of order 2). Since there is no privileged $y$-coordinate, all the elements with $\varphi(y_1)=y_i$ ($i\neq 1$)
have order 4 and $\Gal(K_3/K)$ is the quaternion group $Q_8$ with generators of order 4
\[  \varphi_2 \left\{ \begin{array}{rcl} y_1 & \mapsto & y_2 \\
\ &\ & \ \\
\sqrt{\c} & \mapsto & \sqrt{\c} \end{array}\right. \ ,\
\varphi_3 \left\{ \begin{array}{rcl} y_1 & \mapsto & y_3 \\
\ &\ & \ \\
\sqrt{\c} & \mapsto & -\sqrt{\c} \end{array}\right. \ {\rm and}\ \
\varphi_4 \left\{ \begin{array}{rcl} y_1 & \mapsto & y_4 \\
\ &\ & \ \\
\sqrt{\c} & \mapsto & -\sqrt{\c} \end{array}\right.  \]
\noindent and the element of order 2
\[ \varphi_1 \left\{ \begin{array}{rcl} y_1 & \mapsto & -y_1 \\
\ &\ & \ \\
\sqrt{\c} & \mapsto & \sqrt{\c} \end{array}\right.  .\]

\noindent\underline{If $d=6$}: then $K_3$ contains the cubic extension $K(\sqrt[3]{\D})$ and it must contain its Galois closure too.
Hence if $\z_3\in K$, we have $\Gal(K_3/K)\simeq \Z/3\Z\times \Z/2\Z$; otherwise
$K_3=K(\sqrt[3]{\D},\z_3)$ with $\Gal(K_3/K)\simeq S_3\,$.

\noindent\underline{If $d=4$}:  then there are three cases \begin{itemize}
\item[{\bf a.}] $K=K(\sqrt[3]{\D})$ and $K_3=K(\sqrt{\d})$, or
\item[{\bf b.}] $K=K(\sqrt[3]{\D})$ and $K_3=K(\sqrt{\c},y_1)$, or
\item[{\bf c.}] $K=K(\sqrt{\c})$ and $K_3=K(y_1)$.
\end{itemize}

\noindent In all these cases $K_3$ contains a quadratic subextension $K'$ which is either $K(\sqrt{\c})$ (cases {\bf a} and {\bf b})
or $K(\sqrt{\d})$ (case {\bf c}). If $\z_3\notin K'$ then $K_3=K'(\z_3)$ and $\Gal(K_3/K')\simeq \Z/2\Z\times \Z/2\Z$.
If $\z_3\in K'$, then $K'$ is the unique quadratic subextension, $\Gal(K_3/K)$ is isomorphic to $\Z/4\Z$ and it is generated by
\[ \varphi_{\bf a} \left\{ \begin{array}{rcl} \sqrt{\c} & \mapsto & -\sqrt{\c} \\
\ &\ & \ \\
\sqrt{\d} & \mapsto & \sqrt{\d'} \end{array}\right.  \ ,\
\varphi_{\bf b} \left\{ \begin{array}{rcl} \sqrt{\c} & \mapsto & -\sqrt{\c} \\
\ &\ & \ \\
y_1 & \mapsto & y_3 \end{array}\right. \ {\rm or}\
\varphi_{\bf c}  \left\{ \begin{array}{rcl} \sqrt{\d} & \mapsto & -\sqrt{\d} \\
\ &\ & \ \\
y_1 & \mapsto & y_2 \end{array}\right. \ . \]

\section{Fields $K(\E[4])$}\label{Secm=4}
This section focuses on the case $m=4$ (we remark that the $\gamma$ and $\delta$ here have no relation with
the same symbols appearing in Section \ref{Secm=3}). Let $K$ be a field, with ${\rm char}(K)\neq 2,3$, and let
$\E$ be an elliptic curve defined over $K$, with Weierstrass form $y^2=x^3+Ax+B$.
The roots $\alpha$, $\beta$ and $\gamma$ of $x^3+Ax+B=0$ are the $x$-coordinates of the points of order 2 of $\E$.
In particular $\alpha+\beta+\gamma=0$.
The points of exact order $4$ of $\E$ are $\pm P_1$, $\pm P_2$, $\pm P_3$, $\pm P_4$, $\pm P_5$, $\pm P_6$, where
\[ \begin{split}
&  P_1=(\alpha+\sqrt{(\alpha-\beta)(\alpha-\gamma)},(\alpha-\beta)\sqrt{\alpha-\gamma}+(\alpha-\gamma)\sqrt{\alpha-\beta}),\\
&  P_2=(\beta+\sqrt{(\beta-\alpha)(\beta-\gamma)},(\beta-\gamma)\sqrt{\beta-\alpha} +(\beta-\alpha)\sqrt{\beta-\gamma}),\\
&  P_3=(\alpha-\sqrt{(\alpha-\beta)(\alpha-\gamma)}, (\alpha-\beta)\sqrt{\alpha-\gamma}-(\alpha-\gamma)\sqrt{\alpha-\beta}),\\
&  P_4=(\beta-\sqrt{(\beta-\alpha)(\beta-\gamma)}, (\beta-\alpha)\sqrt{\beta-\gamma}-(\beta-\gamma)\sqrt{\beta-\alpha}),\\
&  P_5=\left(\gamma+\sqrt{(\alpha-\gamma)(\beta-\gamma)},
\frac{(\alpha-\gamma)(\beta-\gamma)}{\sqrt{\gamma-\alpha}}+\frac{(\alpha-\gamma)(\beta-\gamma)}{\sqrt{\gamma-\beta}}\right),\\
&  P_6=\left(\gamma-\sqrt{(\alpha-\gamma)(\beta-\gamma)},
\frac{(\alpha-\gamma)(\beta-\gamma)}{\sqrt{\gamma-\alpha}}-\frac{(\alpha-\gamma)(\beta-\gamma)}{\sqrt{\gamma-\beta}}\right).\\
\end{split} \]
We take $P_1$ and $P_2$ as basis of the $4$-torsion subgroup
of $\E$. With the explicit formulas for the coordinates of the $4$-torsion points its easy to check
that (see, for example, \cite{DZ2})
\[ K_4=K(\sqrt{-1},\sqrt{\alpha-\beta},\sqrt{\beta-\gamma}, \sqrt{\gamma-\alpha}) \ .\]
\noindent Another quick way to find this extension is by
applying Theorem \ref{zetaversusy}.

\subsection{The degree $[K_4:K]$}
By definition $K(\alpha,\beta)$ is the splitting field of $x^3+Ax+B$, i.e., the field generated by the 2-torsion points.
Hence $[K(\alpha,\beta):K]=[K_2:K]\leqslant 6$.
Then $K_4=K(\sqrt{\alpha-\beta}, \sqrt{\alpha-\gamma}, \sqrt{\beta-\gamma},\sqrt{-1})$ has degree
at most $16\cdot [K(\alpha,\beta):K]\leqslant 96$ which is, as expected, the cardinality of $\GL_2(\Z/4\Z)$.
As mentioned at the beginning of Section \ref{SecdegK3}, once we find a curve $\E$ defined over $\Q$ with
$[\Q(\E[4]):\Q]=96$ (see Proposition \ref{Propd=96} below), we know that any degree $d$ dividing 96 is obtainable
over some number field $K$.

\begin{theorem} \label{classification_m=4}
With notations as above, put $d':=[K_2:K]$ and $d:=[K_4:K]$. Consider the conditions
\[ \begin{array}{ll}
 {\bf A1.}\ \sqrt{\alpha-\beta}\notin K_2 \ ,& {\bf A3.}\ \sqrt{\beta-\gamma}\notin K_2(\sqrt{\alpha-\beta},\sqrt{\alpha-\gamma})\ , \\
 {\bf A2.}\ \sqrt{\alpha-\gamma}\notin K_2(\sqrt{\alpha-\beta})\ , & {\bf A4.}\ \sqrt{-1}\notin
K(\sqrt{\alpha-\beta},\sqrt{\alpha-\gamma},\sqrt{\beta-\gamma})\ .\end{array}\]

\noindent Then the degrees are the following
\begin{center}\begin{tabular}{|c|c|c|c|c|c|}
\hline
$d$ & $d'$ & {\em holding conditions} & $d$ & $d'$ & {\em holding conditions}  \\
\hline
{\em 96} & {\em 6} & {\bf A1}, {\bf A2}, {\bf A3}, {\bf A4} & {\em 12} & {\em 3} & 2 among {\bf A1}, {\bf A2}, {\bf A3}, {\bf A4} \\
\hline
{\em 48} & {\em 6} & 3 among {\bf A1}, {\bf A2}, {\bf A3}, {\bf A4} & {\em 8} & {\em 2} & 2 among {\bf A1}, {\bf A2}, {\bf A3}, {\bf A4} \\
\hline
{\em 48} & {\em 3} & {\bf A1}, {\bf A2}, {\bf A3}, {\bf A4} & {\em 8} & {\em 1} & 3 among {\bf A1}, {\bf A2}, {\bf A3}, {\bf A4} \\
\hline
{\em 32} & {\em 2} & {\bf A1}, {\bf A2}, {\bf A3}, {\bf A4} & {\em 6} & {\em 6} & none \\
\hline
{\em 24} & {\em 6} & 2 among {\bf A1}, {\bf A2}, {\bf A3}, {\bf A4} & {\em 6} & {\em 3} & 1 among {\bf A1}, {\bf A2}, {\bf A3}, {\bf A4} \\
\hline
{\em 24} & {\em 3} & 3 among {\bf A1}, {\bf A2}, {\bf A3}, {\bf A4} & {\em 4} & {\em 2} & 1 among {\bf A1}, {\bf A2}, {\bf A3}, {\bf A4} \\
\hline
{\em 16} & {\em 2} & 3 among {\bf A1}, {\bf A2}, {\bf A3}, {\bf A4} & {\em 4} & {\em 1} & 2 among {\bf A1}, {\bf A2}, {\bf A3}, {\bf A4} \\
\hline
{\em 16} & {\em 1} & {\bf A1}, {\bf A2}, {\bf A3}, {\bf A4} & {\em 3} & {\em 3} & none \\
\hline
{\em 12} & {\em 6} & 1 among {\bf A1}, {\bf A2}, {\bf A3}, {\bf A4} & {\em 2} & {\em 2} & none \\
\hline
 &  &  & {\em 2} & {\em 1} & 1 among {\bf A1}, {\bf A2}, {\bf A3}, {\bf A4} \\
\hline
\end{tabular}\end{center}
\end{theorem}

\begin{proof} Computations are straightforward (every condition provides a degree 2 extension).\end{proof}

\noindent We show that any degree $d$ is obtainable by providing a rather general case over $\Q$ with $d=96$.
To stay coherent with our previous notations we set $\Q(\E[4])=:\Q_4$ and $\Q(\E[2])=:\Q_2$
(not to be confused with the $2$-adic field).

\begin{proposition}\label{Propd=96}
Assume that $x^3+Ax+B\in\Q[x]$ is irreducible, that $\Delta=-16(27B^2+4A^3)$ is positive and not a square in $\Q$ and
that $\alpha$, $\beta$ and $\gamma$ are pairwise distinct real numbers. Then $[\Q_4:\Q]=96$.
\end{proposition}

\begin{proof}
Put $\delta=-3\alpha^2-4A$ and note that, once $\alpha$ is fixed the other two roots are
$\ds{\frac{-\alpha \pm \sqrt{\delta}}{2}}$. By renaming the three roots (if necessary), we may assume that $\alpha>\beta>\gamma$, so that
all the generators except $\sqrt{-1}$ are real and
\begin{equation}\label{d=96Eq1}
\begin{array}{rl} [\Q_4:\Q] & = 2[\Q(\sqrt{\alpha-\beta},\sqrt{\alpha-\gamma},\sqrt{\beta-\gamma}):\Q]\\
& \\
\ & = 2[\Q\left(\sqrt{\ds{\frac{3\alpha+\sqrt{\delta}}{2}}},\sqrt{\ds{\frac{3\alpha-\sqrt{\delta}}{2}}},\sqrt[4]{\delta}\right):\Q]\ .
\end{array}
\end{equation}
By the choice of $\alpha$,
we have that $A<0$ and  the polynomial $x^3+Ax+B$ has a minimum in $x=\sqrt{\ds{-\frac{A}{3}}}$.  Hence
$\alpha>\sqrt{\ds{-\frac{A}{3}}} \hspace{0.3cm} \textrm{ and in particular }  \hspace{0.3cm} 3\alpha^2+A>0.$

\noindent By the hypotheses, we have that $[\Q_2:\Q]=[\Q(\alpha,\sqrt{\delta}):\Q]=6$ and $\delta>0$ is not a square in $\Q(\alpha)$.
Obviously $[\Q_2(\sqrt[4]{\delta}):\Q_2]=2$; moreover $\ds{\frac{3\alpha+\sqrt{\delta}}{2}}$ is a square in $\Q_2$ if and only
if $\ds{\frac{3\alpha-\sqrt{\delta}}{2}}$ has the same property. Assume $\ds{\frac{3\alpha+\sqrt{\delta}}{2}}\in (\Q_2^*)^2\,$, i.e.,
$\ds{\frac{3\alpha+\sqrt{\delta}}{2}}=(a+b\sqrt{\delta})^2$, for some $a,b\in \Q_2\,$. Then
\[\left\{ \begin{array}{l} a^2+b^2\delta=\ds{\frac{3\alpha}{2}} \\
2ab=\ds{\frac{1}{2}} \end{array}\right. \Longrightarrow \left\{ \begin{array}{l} a^2+\ds{\frac{\delta}{16a^2}}=\ds{\frac{3\alpha}{2}} \\
b=\ds{\frac{1}{4a}} \end{array}\right.\ , \]
leading to
\[ a^2 = \frac{12\alpha\pm\sqrt{144\alpha^2-16\delta}}{16}=\frac{3\alpha\pm\sqrt{9\alpha^2-\delta}}{4}\in \Q(\alpha)\ .\]
Hence $9\alpha^2-\delta=12\alpha^2+4A$ must be a square in $\Q(\alpha)$, i.e., $3\alpha^2+A\in(\Q(\alpha)^*)^2\,$.
Let $N$ denote the norm map from $\Q(\alpha)$ to $\Q$. Then $N(3\alpha^2+A)=27B^2+4A^3$ is not a square in $\Q$ by hypothesis
and this contradicts $3\alpha^2+A\in(\Q(\alpha)^*)^2\,$. Therefore
\[ [\Q_2\left(\ds{\sqrt{\frac{3\alpha+\sqrt{\delta}}{2}}}\right):\Q_2]=
[\Q_2\left(\ds{\sqrt{\frac{3\alpha-\sqrt{\delta}}{2}}}\right):\Q_2]=2 \ \]
and we have to prove that the three quadratic extensions of $\Q_2$ we found are independent.

\noindent The elements $\sqrt{\ds{\frac{3\alpha+\sqrt{\delta}}{2}}}$ and $\sqrt{\ds{\frac{3\alpha-\sqrt{\delta}}{2}}}$ generate the same
quadratic extension over $\Q_2$ if and only if
\[ \ds{\frac{3\alpha+\sqrt{\delta}}{2}}\,\cdot\,\ds{\frac{2}{3\alpha-\sqrt{\delta}}} = \frac{9\alpha^2-\delta}{(3\alpha-\sqrt{\delta})^2}
\in (\Q_2^*)^2 \ ,\]
i.e., if and only if $3\alpha^2+A\in (\Q_2^*)^2\,$. We have already seen that $3\alpha^2+A\not\in (\Q(\alpha)^*)^2$, so we must have
$3\alpha^2+A=(a+b\sqrt{\delta})^2$ with $a,b\in\Q(\alpha)$ and $b\neq 0$. A little computation gives
\[ b^2=-\frac{3\alpha^2+A}{3\alpha^2+4A}  \in (\Q(\alpha)^*)^2 \ ,\]
but
\[ N\left(-\frac{3\alpha^2+A}{3\alpha^2+4A}\right) = -1 \not\in (\Q^*)^2 \]
and this is a contradiction. Hence
\[ [\Q_2\left( \sqrt{\ds{\frac{3\alpha+\sqrt{\delta}}{2}}},\sqrt{\ds{\frac{3\alpha-\sqrt{\delta}}{2}}} \right):\Q_2]=4\ .\]
Now $\sqrt[4]{\delta}$ and $\sqrt{\ds{\frac{3\alpha\pm\sqrt{\delta}}{2}}}$ generate the same quadratic extension of $\Q_2$ if and
only if
\[ \ds{\frac{3\alpha\pm\sqrt{\delta}}{2}}\,\cdot\,\frac{1}{\sqrt{\delta}} = \frac{6\alpha\sqrt{\delta}\pm 2\delta}{4\delta}
\in (\Q_2^*)^2 \,\]
i.e., if and only if $  6\alpha\sqrt{\delta}\pm 2\delta = (a+b\sqrt{\delta})^2$ for some $a,b\in\Q(\alpha)$.
This leads to\begin{itemize}
\item[{\bf 1.}] $a^2+b^2\delta=2\delta$ and $2ab=6\alpha$: solving for $a$ we get
\[ a^2= \delta\pm\sqrt{\delta^2-9\alpha^2\delta} \in \Q(\alpha) \ .\]
Hence
\[ \delta^2-9\alpha^2\delta = (-3\alpha^2-4A)(-12\alpha^2-4A) \in (\Q(\alpha)^*)^2 \,\]
i.e., $(3\alpha^2+4A)(3\alpha^2+A) \in (\Q(\alpha)^*)^2$. But by hypothesis $3\alpha^2+4A=-\delta<0$ and we recall that
$3\alpha^2+A>0$; thus $(3\alpha^2+4A)(3\alpha^2+A)<0$ cannot be a square in the real field $\Q(\alpha)$.
\item[{\bf 2.}] $a^2+b^2\delta=-2\delta$ and $2ab=6\alpha$: this is impossible because $a^2+b^2\delta >0$, while
$-2\delta<0$.
\end{itemize}
Then
\[ [\Q_2\left(\sqrt[4]{\delta},\sqrt{\frac{3\alpha+\sqrt{\delta}}{2}}\right):\Q_2]=
[\Q_2\left(\sqrt[4]{\delta},\sqrt{\frac{3\alpha-\sqrt{\delta}}{2}}\right):\Q_2] = 4\ .\]
With similar computations one checks that the extension generated by $\sqrt[4]{\delta}$ is also independent
from the third quadratic extension contained in
$\Q_2\left(\sqrt{\frac{3\alpha+\sqrt{\delta}}{2}},\sqrt{\frac{3\alpha-\sqrt{\delta}}{2}}\right)$, which is
$\Q_2(\sqrt{3\alpha^2+A})$. Hence
\[ [ \Q_2\left(\sqrt{\frac{3\alpha+\sqrt{\delta}}{2}},\sqrt{\frac{3\alpha-\sqrt{\delta}}{2}},\sqrt[4]{\delta}\right):
\Q]=48 \]
and, by \eqref{d=96Eq1}, we have $[\Q_4:\Q]=96$.
\end{proof}

\noindent With reducible polynomials $x^3+Ax+B$ we can easily obtain examples of smaller degrees, in particular
when $A=0$ or $B=0$ (obviously, since $\sqrt{-1}\in \Q_4$, we cannot obtain extension of degree 1 or 3 over $\Q$).

\begin{example}
The curve
\[ y^2=x^3-\frac{481}{3}x+\frac{9658}{27}=\left(x-\frac{34}{3}\right)\left(x-\frac{7}{3}\right)\left(x+\frac{41}{3}\right) \]
provides $\sqrt{\alpha-\beta}=3$, $\sqrt{\alpha-\gamma}=5$ and $\sqrt{\beta-\gamma}=4$. Then
$\Q_4=\Q(\sqrt{-1})$ has degree 2 over $\Q$.\\
The curve
\[ y^2=x^3-22x-15=(x-5)(x^2+5x+3) \]
yields
\[ \Q_2=\Q(\sqrt{13})\quad{\rm and}\quad
\Q_4=\Q\left(\sqrt{\frac{5+\sqrt{13}}{2}},\sqrt{\frac{5-\sqrt{13}}{2}},\sqrt[4]{5},\sqrt{-1}\right) \]
which has degree 32 over $\Q$.
\end{example}

\begin{proposition}\label{A=0B=0}
If $A=0$, then $\Q_4=\Q(\z_{12},\sqrt{\sqrt[3]{B}(1-\z_3)})$ and
\[ [\Q_4:\Q]=\left\{\begin{array}{ll} 8 & {\rm if}\ B\in (\Q^*)^3\ ,\\
24 & {\rm otherwise} \ .\end{array} \right. \]
If $B=0$, then $\Q_4=\Q(\sqrt{2},\sqrt{-1},\sqrt[4]{-A})$ and
\[ [\Q_4:\Q]=\left\{\begin{array}{ll} 16 & {\rm if}\ A\neq \pm 2a^2, \pm a^2\ {\rm with}\ a\in\Q \ ,\\
8 & {\rm if}\ A=\pm 2a^2 \ {\rm with}\ a\in\Q \ ,\\
4 & {\rm if}\ A= a^4, \pm 4a^4 \ {\rm with}\ a\in\Q \ ,\\
8 & {\rm otherwise} \ .\end{array} \right. \]
\end{proposition}

\begin{proof}
For $A=0$ just take $\alpha=\sqrt[3]{B}$, $\beta=\z_3\sqrt[3]{B}$ and $\gamma=\z_3^2\sqrt[3]{B}$ to get
\[ \Q_4=\Q\left(\z_3,\sqrt{-1},\sqrt{\sqrt[3]{B}(1-\z_3)},\sqrt{\sqrt[3]{B}(1-\z_3^2)},
\sqrt{\sqrt[3]{B}(\z_3-\z_3^2)}\right) \ .\]
Obviously $\Q(\z_3,\sqrt{-1})=\Q(\z_{12})$, moreover, the elements $\sqrt{\sqrt[3]{B}(1-\z_3)}$,
$\sqrt{\sqrt[3]{B}(1-\z_3^2)}$ and $\sqrt{\sqrt[3]{B}(\z_3-\z_3^2)}$ generate the same extension
of $\Q(\z_{12})$. Therefore
\[ \Q_4=\Q\left(\z_{12},\sqrt{\sqrt[3]{B}(1-\z_3)}\right) \ \]
and the first statement follows.

\noindent For $B=0$ let $\alpha=0$, $\beta=\sqrt{-A}$ and $\gamma=-\beta$ to get
$\Q_4=\Q(\sqrt[4]{-A},\sqrt{2},\sqrt{-1})$.
The unique quadratic subfield of $\Q(\sqrt[4]{-A})$ is $\Q(\sqrt{-A})$, hence, if
$\Q(\sqrt{-A})\neq \Q(\sqrt{\pm 2})$, $\Q(\sqrt{-1})$, $\Q$, i.e., if $A\neq \pm 2a^2, \pm a^2$
for some $a\in \Q$, we have $[\Q_4:\Q]=16$. The remaining cases are straightforward.
\end{proof}

\subsection{Galois groups}
One can find descriptions for $\GL_2(\Z/4\Z)$ in \cite[Section 5.1]{Ad} or \cite[Section 3]{Ho}: the most suitable for our
goals is the exact sequence coming from the canonical projection $\GL_2(\Z/4\Z) \rightarrow \GL_2(\Z/2\Z)$, whose kernel
we denote by $H^4_2\,$. Obviously
\[ H^4_2 = \left\{ \left( \begin{array}{cc} 1+2a & 2b \\ 2c & 1+2d \end{array} \right)\in \GL_2(\Z/4\Z)\right\} \]
and it is easy to check that it is an abelian group of order 16 and exponent 2, i.e., isomorphic to $(\Z/2\Z)^4\,$.
 By sending the row $(1\ 1)$ to $(3\ 3)$ and leaving rows $(1\ 0)$ and $(0\ 1)$ fixed, we see that there exists a section\,
$\GL_2(\Z/2\Z) \rightarrow \GL_2(\Z/4\Z)$ which splits the sequence
\[ H^4_2 \hookrightarrow \GL_2(\Z/4\Z) \twoheadrightarrow \GL_2(\Z/2\Z) \]
as a semi-direct product. For any $K$, we have a commutative diagram
\[ \xymatrix{ H^4_2 \ar@{^(->}[r] \ar@{->>}[d] & \GL_2(\Z/4\Z) \ar@{->>}[r] \ar@{->>}[d] & \GL_2(\Z/2\Z) \ar@{->>}[d] \\
\Gal(K_4/K_2) \ar@{^(->}[r] & \Gal(K_4/K) \ar@{->>}[r] & \Gal(K_2/K) \ .} \]
The structure of $\Gal(K_4/K)$ can be derived from the lower sequence (which splits as well), checking the conditions
of Theorem \ref{classification_m=4} to compute $d'$ (which identifies $\Gal(K_2/K)$ as one among $\Id$, $\Z/2\Z$,
$\Z/3\Z$ or $S_3\,$) and the $i\in\{0,\dots,4\}$ for which $\Gal(K_4/K_2)\simeq (\Z/2\Z)^i\,$.

\end{document}